\documentclass{article}

\usepackage[preprint]{neurips_2026}

\usepackage[utf8]{inputenc} % allow utf-8 input
\usepackage[T1]{fontenc}    % use 8-bit T1 fonts
\usepackage{hyperref}       % hyperlinks
\usepackage{url}            % simple URL typesetting
\usepackage{booktabs}       % professional-quality tables
\usepackage{nicefrac}       % compact symbols for 1/2, etc.
\usepackage{microtype}      % microtypography
\usepackage{xcolor}         % colors
\makeatletter
\let\MainOriginalNoticeString\@noticestring
\renewcommand{\@noticestring}{{\MainOriginalNoticeString}}
\makeatother
\usepackage{bm}
\usepackage{amsmath,amsfonts,amsthm}
\usepackage{subcaption}
\usepackage{latexsym}
\usepackage{mathrsfs}
\usepackage{float}
\usepackage{graphicx}
\usepackage{cleveref}
\usepackage{enumitem}
\usepackage{tcolorbox}
\usepackage{array}

\newtheorem{lemma}{Lemma}
\newtheorem{theorem}{Theorem}
\newtheorem{corollary}{Corollary}
\newtheorem{defi}{Definition}
\newtheorem{proposition}{Proposition}

\newtheorem{fact}{Fact}
\newtheorem{remark}{Remark}

\def\R{\mathbb{R}}

\title{A Local Geometric Analysis of Maximal Coding Rate Reduction via Error Bounds}

\author{%
  Peng Wang\thanks{The first two authors contributed equally to this work.} \\
Department of Artificial Intelligence \& \\ Institute of Artificial Intelligence and Brain Sciences\\
  University of Macau\\
  \texttt{peng8wang@gmail.com} \\
  \And
  Po Chen$^*$ \\
  School of Data Science \\
  Fudan University\\
  \texttt{chenp24@m.fudan.edu.cn} \\
  \AND
  Huikang Liu \\
  Antai College of Economics and Management\\
  Shanghai Jiao Tong University  \\
  \texttt{hkl1u@sjtu.edu.cn} \\
  \And
  Rujun Jiang\thanks{Corresponding author.} \\
  School of Data Science \\
  Fudan University \\
  \texttt{rjjiang@fudan.edu.cn} \\
}

\begin{document}

\maketitle

\begin{abstract}
While the maximal coding rate reduction (MCR$^2$) objective has become a useful principle for learning compact and discriminative representations, a theoretical understanding of its optimization geometry remains limited, especially on why first-order methods often exhibit fast convergence near its local maximizers. In this work, we show that the MCR$^2$ objective satisfies an error-bound property in a neighborhood of its local maximizers. This property guarantees that the distance to the local maximizer set can be bounded above by the gradient norm at any point in the neighborhood, directly implying regularity conditions such as the Polyak--\L ojasiewicz inequality and quadratic growth. Leveraging this property, we prove that first-order methods converge linearly to a local maximizer of the MCR$^2$ objective under mild conditions. Numerical experiments validate the predicted local convergence behavior and provide empirical evidence that optimizing the MCR$^2$ objective improves representation geometry while maintaining competitive downstream accuracy. 
\end{abstract}

\section{Introduction}\label{sec:intro}

In deep representation learning, the maximal coding rate reduction (MCR$^2$) objective \citep{ma2007segmentation,yu2020learning}, grounded in the lossy coding rate, has attracted increasing attention in recent years. Thanks to the remarkable performance in learning compact and structured representations that facilitate downstream tasks, MCR$^2$-based approaches have been widely applied to multi-class classification \citep{wen2025towards,yu2020learning}, image segmentation \citep{chan2022redunet,Meng2025_TR2C,yu2023emergence}, and image and language generation \citep{yu2024white,wang2025attentiononly}. In particular, neural network architectures obtained by unrolling algorithmic updates for optimizing the MCR$^2$ objective, such as ReduNet \citep{chan2022redunet}, CRATE \citep{yu2024white}, and their variants \citep{yang2024scaling,pai2024masked,wu2025token}, have been demonstrated to work effectively in practice.  

As a guiding principle for learning linear discriminative representations, the MCR$^2$ objective has achieved remarkable empirical success. Nevertheless, due to the highly complicated and non-convex nature of the objective (see Problem \eqref{eq:MCR}), the theoretical understanding of its optimization properties remains limited. A seminal work \citep{chan2022redunet} proved that the global solutions of the MCR$^2$ problem are within-class compressible, between-class discriminative, and maximally diverse. Recently, \cite{wang2024global} explicitly characterized the set of local maximizers of the regularized MCR$^2$ problem. Moreover, they showed that this problem has a benign loss landscape in the sense that every critical point is either a local maximizer or a strict saddle point (see \Cref{fact:land}). More interestingly, they observed that gradient ascent applied to this problem can achieve linear convergence, but a formal theoretical justification remains elusive. In addition, MCR$^2$-based neural network architectures have been shown to be effective at learning feature representations \citep{chan2022redunet, wang2025attentiononly, yu2024white}. These observations naturally raise the question of whether the MCR$^2$ objective possesses favorable optimization properties that ensure efficient convergence when optimizing it, while effectively learning meaningful feature representations when unrolled as a neural network. 

\vspace{-0.05in}
\subsection{Main Contributions}

% This paper complements prior landscape characterizations of the MCR$^2$ objective (see \cite{chan2022redunet,wang2024global}) by establishing a local regularity result around the set of local maximizers. In particular, we prove an error-bound property around this set and use it to explain why first-order methods exhibit local linear convergence. We also make the scope of the experiments explicit: the main experiments test the convergence quantities predicted by the theory, while the extension experiments provide a limited representation-facing check through linear probing, optimizer comparisons, and class-mean feature geometry.

% In this work, we study the optimization geometry of the MCR$^2$ objective (see Problem \eqref{eq:MCR}). Owing to the benign loss landscape established in \cite{wang2024global}, every critical point of Problem \eqref{eq:MCR} is either a local maximizer or a strict saddle. Moreover, \cite{lee2016gradient,lee2019first} show that first-order methods with random initialization almost surely avoid strict saddles. Combined together, these facts imply that, for Problem \eqref{eq:MCR}, randomly initialized first-order methods converge almost surely to local maximizers. We therefore focus on the local-maximizer set in the subsequent analysis. 
In this work, we study the local optimization geometry of the MCR$^2$ objective around its local maximizers. This local perspective is important: since prior work shows that first-order methods with random initialization almost surely avoid strict saddles \citep{lee2016gradient,lee2019first}, the key question is not only where the iterates may converge, but also how regular the landscape is once the trajectory enters a basin of a local maximizer. Our work addresses this question by establishing local regularity of the MCR$^2$ landscape and using it to explain the fast local convergence of first-order methods near local maximizers. Our contributions are summarized as follows:

\begin{itemize}[leftmargin=*]

    \item[$\bullet$] Despite the complexity and inherent non-convexity of the MCR$^2$ problem, we prove that it satisfies a local error-bound property \citep{luo1993error,pang1997error,zhou2017unified} in a neighborhood of its local maximizer set; see \Cref{thm:EB}. Specifically, we show that for any point $\bm Z$ in this neighborhood, the distance from $\bm Z$ to the local maximizer set is controlled by the gradient norm. This error-bound property captures favorable local geometric regularity of the objective and further implies standard regularity conditions, including the Polyak--Lojasiewicz (P\L) inequality \citep{karimi2016linear,polyak1963gradient} and quadratic growth \citep{bonnans1995second,drusvyatskiy2018error}.
    
    \item[$\bullet$] Leveraging the error-bound property and the benign loss landscape of the MCR$^2$ problem (see \Cref{fact:land}), we show that first-order methods for optimizing the MCR$^2$ objective converge to a local maximizer at a local linear rate (see \Cref{coro:rate}). This result explains the fast local convergence of gradient-based methods observed in practice and shows that, once the iterates enter a neighborhood of the local maximizer set, the MCR$^2$ objective has sufficient regularity to ensure geometric decay of the optimization error. It further provides theoretical support for MCR$^2$-based representation learning methods and unrolled architectures \citep{chan2022redunet,wang2025attentiononly,wu2025simplifying,pai2024masked}, where first-order updates are used as building blocks for learning discriminative feature representations.

\end{itemize}

To support our theory, we conduct numerical experiments on both synthetic and real-world datasets. The results corroborate the predicted local convergence behavior and further illustrate the connection between optimizing the MCR$^2$ objective and learning discriminative feature representations.

\subsection{Related Works}

\vspace{-0.05in}
\paragraph{Deep representation learning using MCR$^2$.} %By promoting within-class compressibility and between-class diversity, MCR$^2$-based approaches are powerful for learning discriminative and structure representation in deep learning. 
Current approaches using MCR$^2$ for learning representations can be broadly categorized into two types. The first uses MCR$^2$ as a training objective to learn discriminative and compact representations for different tasks. For example, an early work \citep{ma2007segmentation} demonstrated the effectiveness of the MCR$^2$ objective in classical clustering problems. Recently, \citet{yu2020learning} optimized the MCR$^2$ objective to train a neural network for multi-class classification tasks. The work of \citet{pai2023pursuit} leverages MCR$^2$ to learn discriminative representations for data drawn from multiple subspaces by formulating the problem as a sequential game. The work of \citet{Meng2025_TR2C} applies MCR$^2$ in temporal rate reduction clustering to segment human motion sequences.

The second approach is to unroll neural networks based on algorithmic updates for optimizing the MCR$^2$ objective or its variants. For example, \citet{yu2020learning} proposed a neural network named ReduNet by unrolling gradient ascent for optimizing the MCR$^2$ objective, enabling interpretable and efficient representation learning. Recently, \citet{yu2024white} designed a transformer-like network architecture named CRATE by unrolling alternating minimization for optimizing the sparse rate reduction objective. Using different objectives or different optimization algorithms, \citet{wang2025attentiononly} proposed an attention-only transformer-like architecture, while \citet{wu2025token} developed a linear-time variant of a transformer-like architecture. Please refer to \citep{pai2024masked,wang2025attentiononly,wu2025simplifying,yang2024scaling} for additional neural network architecture designs based on the MCR$^2$ objective. % \cite{wang2025attentiononly} propose an attention-based architecture that compresses all attention operations by selectively contracting a subset

\vspace{-0.05in}
\paragraph{Error bound analysis of non-convex problems.} Error bounds have been extensively studied as a fundamental tool for characterizing the local geometry of optimization problems and establishing convergence guarantees, such as {\em linear} convergence and {\em superlinear} convergence rates, for iterative algorithms \citep{luo1993error,yue2019quadratic,zhou2017unified}. While early works focused primarily on convex settings \citep{luo1993error,zhou2017unified}, recent studies have shown that various classes of non-convex problems also satisfy error-bound properties under structural assumptions. For example, \cite{Liu2019} showed that quadratic programs with orthogonality constraints satisfy an error-bound property, and \cite{wang2023linear} showed the same result for linear programs with orthogonality constraints. \cite{jiang2022holderian} proved an error bound for the trust region subproblem. Recently, \cite{chen2025error} established error bounds for the regularized loss of deep matrix factorization problems. As demonstrated in these works, establishing error-bound properties generally requires dedicated analyses and advanced technical tools. 

% \begin{table}[t]
% \centering
% \caption{Positioning relative to closely related work on MCR$^2$ and nonconvex regularity.}
% \label{tab:related_work_comparison}
% \footnotesize
% \setlength{\tabcolsep}{3.8pt}
% \renewcommand{\arraystretch}{1.08}
% \begin{tabular*}{\textwidth}{@{\extracolsep{\fill}}>{\raggedright\arraybackslash}p{0.25\textwidth}cccc@{}}
% \toprule
% Work & Rep. learning & Landscape & EB & Linear rate \\
% \midrule
% \cite{yu2020learning} & Yes & No & No & No \\
% \cite{chan2022redunet} & Yes & Global optima & No & No \\
% \cite{wang2024global} & Yes & Strict saddles & No & Observed \\
% \midrule
% General EB analyses \cite{luo1993error,zhou2017unified,Liu2019,wang2023linear} & No & Problem dependent & Yes & Yes \\
% \midrule
% \textbf{This paper} & \textbf{Yes} & \textbf{Uses MCR$^2$ landscape} & \textbf{Yes} & \textbf{Yes} \\
% \bottomrule
% \end{tabular*}
% \vspace{-0.06in}
% \end{table}

\vspace{-0.05in}
\paragraph{Notation.} 
For a matrix $\bm A \in \R^{m\times n}$, $\|\bm A\|$ denotes its spectral norm, $\|\bm A\|_F$ its Frobenius norm, and $a_{ij}$ its $(i,j)$-th entry.
For a vector $\bm a \in \R^d$, $\|\bm a\|$ denotes its $\ell_2$ norm, $a_i$ its $i$-th entry, and $\mathrm{diag}(\bm a)$ the diagonal matrix with $\bm a$ on the diagonal. Given a positive integer $n$, we denote by $[n]$ the set $\{1,\ldots,n\}$. We denote by $\mathrm{BlkD}(\bm{X}_1, \dots, \bm{X}_n)$ the block diagonal matrix with diagonal blocks $\bm{X}_1, \dots, \bm{X}_n$. Given a set of numbers $\{a_k\}_{k=1}^K$, let $a_{\max} := \max\{a_k: k \in [K]\}$ and $a_{\min} := \min\{a_k: k \in [K]\}$. Let $\mathcal{O}^{m\times n} := \left\{\bm X \in \R^{m\times n}: \bm X^T\bm X = \bm I_n \right\}$ denote the set of all $m\times n$ orthonormal matrices. In particular, we write $\mathcal{O}^m := \mathcal{O}^{m\times m}$ to denote the set of all $m\times m$ orthogonal matrices. Given a point $\bm X \in \R^{m\times n}$, its distance to a non-empty and closed set $\mathcal{X} \subseteq \R^{m\times n}$ is defined as $\mathrm{dist}(\bm X, \mathcal{X}) := \min_{\bm Y \in \mathcal{X}} \|\bm X - \bm Y\|_F$.% For a function $F(\bm Z)$ with $\bm Z = [\bm Z_1\ \cdots\ \bm Z_K]$, we denote $\nabla_k F(\bm Z) := \frac{\partial F(\bm Z)}{\partial \bm Z_k}$ for each $k \in [K]$.  
% \footnote{We state the definitions for a maximization problem, consistent with Problem~\eqref{eq:MCR}. The minimization analogs follow by applying the definitions to $-F$.}
\begin{defi}
Consider the maximization of a twice continuously differentiable function $F$. We say  \\
(i) $\bm Z^*$ is a critical point if $\nabla F(\bm Z^*) = \bm 0$;\\
(ii) $\bm Z^*$ is a local maximizer if $F(\bm Z^*) \ge F(\bm Z)$ for all $\bm Z$ in some neighborhood of $\bm Z^*$; \\
(iii) $\bm Z^*$ is a strict saddle point if it is a critical point and its Hessian has at least one strictly positive eigenvalue.\footnote{As defined in \cite{lee2019first}, our definition of a {\em strict saddle point} allows for the inclusion of local minima.}     
\end{defi}

% \footnote{For a continuously differentiable function $f:\R^d \to \R$, a critical point $\bm x^*$ is a strict saddle point if $\lambda_{\min}(\nabla^2f(\bm x^*)) < 0$.}

\vspace{-0.05in}
\section{Main Results}\label{sec:main}

% In this section, we first briefly introduce the MCR$^2$ problem in \Cref{subsec:setup} and then present our main theorem in \Cref{subsec:thm}. 

\subsection{Problem Setup}\label{subsec:setup}

To begin, we introduce the MCR$^2$ objective for supervised learning problems. Consider a training dataset $\{\bm x_i\}_{i=1}^N \subseteq \R^n$ drawn from $K \ge 2$ classes. Let $\bm z_i \in \R^d$ be the feature representation of $\bm x_i$ obtained via a feature mapping, such as a neural network, for each $i\in [N]$. Without loss of generality, we assume that the samples are organized in a class-by-class manner. For each $k \in [K]$, let $N_k$ be the number of samples in the $k$-th class and $\bm Z_k \in \R^{d\times N_k}$ be the matrix whose columns are the feature representations of the samples in the $k$-th class. We then define 
\begin{align}\label{eq:Z}
\bm Z := \begin{bmatrix}
\bm Z_1 & \cdots & \bm Z_K
\end{bmatrix} \in \R^{d\times N}. 
\end{align}   
Moreover, let
\begin{align}\label{eq:R}
R(\bm Z) := \frac{1}{2}\log\det\left( \bm I + \frac{d}{N\epsilon^2}\bm Z \bm Z^T \right),\quad R_c(\bm Z_k) := \frac{N_k}{2N}\log\det\left( \bm I + \frac{d}{N_k\epsilon^2}\bm Z_k \bm Z_k^T \right),
\end{align}
for each $k \in [K]$, where $\epsilon > 0$ is a coding precision. Notably, $R(\bm Z)$ approximates the rate-distortion function of a Gaussian source when $\epsilon$ is sufficiently small \citep{cover}. As argued in \citep{chan2022redunet,wang2024global}, $R(\bm Z)$ measures the global coding rate of all features and therefore rewards a representation that spans a large volume, which encourages diversity and separation across classes. In contrast, $R_c(\bm Z_k)$ measures the coding rate within the $k$-th class; subtracting $\sum_{k=1}^K R_c(\bm Z_k)$ penalizes within-class variability and encourages features from the same class to be compressible. Thus, maximizing the difference between the global and classwise coding rates promotes between-class discrimination and within-class compactness. This motivates the following MCR$^2$ problem:    
\begin{align}\label{eq:MCR}
\max_{\bm Z \in \R^{d\times N}}\ F(\bm Z) := R(\bm Z) - \sum_{k=1}^K R_c(\bm Z_k) - \frac{\lambda}{2}\|\bm Z\|_F^2, 
\end{align}
where $\lambda > 0$ is a regularization parameter. %\cite{chan2022redunet,yu2020learning} studied a constrained version of the MCR$^2$ problem, which imposes a norm constraint on $\bm Z$ rather than using a regularizer. \cite{wang2024global} proved the equivalence between the regularized and constrained versions in terms of global optimality. 

\cite[Theorem 1]{wang2024global} has shown that any local optimal solution of this problem satisfies the property that features from the same class are compressible, while features from different classes are orthogonal to each other. Moreover, \cite[Theorem 2]{wang2024global} has proven that Problem \eqref{eq:MCR} admits the benign loss landscape in the sense that each critical point is either a local maximizer or a strict saddle point.  For completeness and to make the subsequent theorem statements self-contained, we recall these two results below. 
\begin{fact}\label{fact:land}
Given the coding precision $\epsilon > 0$, suppose that the regularization parameter $\lambda$ satisfies
\begin{align}\label{eq:lambda}
\lambda \in \left(0, \frac{d(\sqrt{N/N_{\max}}-1)}{N ( \sqrt{N/N_{\max}} + 1)\epsilon^2} \right]. 
\end{align}
(i) Every critical point of Problem \eqref{eq:MCR} is either a local maximizer or a strict saddle point. \\
(ii) It holds that $\bm Z = [\bm Z_1\ \cdots\ \bm Z_K]$ is a local maximizer if and only if each block $\bm Z_k$ admits the following singular value decomposition 
\begin{align}\label{eq:Zk}
\bm Z_k = \sigma_k^* \bm U_k\bm V_k^T,  
\end{align}
where (a) $r_k := \mathrm{rank}(\bm Z_k)$ satisfies $r_k \in [0,\mathrm{min}\{N_k,d\})$ and $\sum_{k=1}^K r_k \le \min\{N,d\}$,
(b) $\bm U_k \in \mathcal{O}^{d \times r_k}$ satisfies $\bm U_k^T\bm U_l = \bm 0$ for all $l \neq k$ and $\bm V_k \in \mathcal{O}^{N_k\times r_k}$ for all $k \in [K]$,\footnote{When $r_k = 0$, we define $\bm U_k=\bm 0$ and $\bm V_k = \bm 0$.} and (c) the singular value  
\begin{align}\label{eq:sigma eta}
& \sigma_k^* := \left(\frac{ \eta_k + \sqrt{\eta_k^2 - 4\lambda^2 N/N_k}}{2\lambda \alpha_k}\right)^{1/2},\quad \eta_k := (\alpha_k - \alpha) - \lambda\left(\frac{N}{N_k} + 1\right),
\end{align}
where $\alpha=d/(N\epsilon^2)$ and $\alpha_k=d/(N_k\epsilon^2)$.
\end{fact}
Based on the above theorem, we denote the set of local maximizers by 
\begin{equation}\label{set:Z}
\begin{aligned}
\mathcal{Z}:=\Big\{&[\bm Z_1\ \cdots\ \bm Z_K]:\;
\bm Z_k=\sigma_k^*\bm U_k\bm V_k^T,\ 
\bm U_k\in\mathcal{O}^{d\times r_k},\ 
\bm V_k\in\mathcal{O}^{N_k\times r_k},\\
&\bm U_k^T\bm U_l=\bm 0,\ \forall k\ne l,\quad
r_k\in[0,\min\{N_k,d\}),\quad
\sum_{k=1}^K r_k\le \min\{N,d\}\Big\}.
\end{aligned}
\end{equation}
In particular, \cite{lee2019first,lee2016gradient} have proved that first-order methods with random initialization, such as gradient ascent, almost surely escape strict saddle points. This, together with \Cref{fact:land}(i), implies that the iterates generated by first-order methods almost surely converge to a local maximizer. Therefore, it suffices to restrict our attention to the local maximizers of Problem \eqref{eq:MCR}, as defined in \eqref{set:Z}, in the subsequent analysis.  

\subsection{Main Theorems}\label{subsec:thm} 

Based on the above setup, we establish the error-bound property for Problem \eqref{eq:MCR} in a neighborhood of its local maximizers in the following theorem, whose proof is deferred to \Cref{app:thm EB}.  
\begin{theorem}\label{thm:EB}
Given the coding precision $\epsilon > 0$, suppose that the regularization parameter $\lambda$ satisfies \eqref{eq:lambda}. There exist constants $\delta, \kappa > 0$ such that for all $\bm Z$ satisfying $\mathrm{dist}(\bm Z,\mathcal{Z}) \le \delta$, it holds that 
\begin{align}\label{eq:EB}
\mathrm{dist}(\bm Z, \mathcal{Z}) \le \kappa \|\nabla F(\bm Z)\|_F, 
\end{align}
where $\cal Z$, defined in \eqref{set:Z}, denotes the set of all local maximizers of Problem \eqref{eq:MCR}.
\end{theorem}
To clarify the theorem’s significance, we make the following comments: 
\begin{itemize}[leftmargin=*]
\item The error-bound property \eqref{eq:EB} shows that, for any point $\bm Z$ in a neighborhood of the local maximizer set $\mathcal Z$, the distance from $\bm Z$ to $\mathcal Z$ is bounded by the gradient norm. This establishes a form of local geometric regularity for the MCR$^2$ objective, ensuring that the landscape around its local maximizers is well behaved despite the non-convexity of the problem. Together with \Cref{fact:land}, this error-bound property provides a refined local picture of the MCR$^2$ landscape, connecting global saddle avoidance with local convergence behavior.

\item As demonstrated in \citep{rebjock2024fast,pmlr-v242-liao24a}, the error-bound property can be leveraged to establish other regularity properties, such as Polyak–Łojasiewicz (P\L) inequality \citep{karimi2016linear,polyak1963gradient} and quadratic growth \citep{bonnans1995second,drusvyatskiy2018error}. Notably, all the above regularity properties, together with some algorithm-dependent properties, can be used to facilitate the convergence-rate analysis of first-order methods. Specifically, as shown in \citep{karimi2016linear,luo1993error,conver,zhou2017unified}, a linear convergence rate for first-order iterative methods can be established when the iterates satisfy the sufficient decrease condition, a cost-to-go estimate, and a safeguard condition.

\item The proof is technically involved because $\mathrm{dist}(\bm Z,\mathcal{Z})$ does not admit a simple closed-form expression in this problem. Unlike error-bound analyses for problems with simpler solution sets \citep{Liu2019,wang2023linear,jiang2019novel}, our argument constructs an intermediate point and separately controls its distance to $\bm Z$ and to $\mathcal{Z}$; see \Cref{pf:EB} and \Cref{sec:proof}.
\end{itemize}

\begin{remark}[Error bounds from representation space to parameter space]
Although \Cref{thm:EB} characterizes the local geometry of $F$ when the representations are treated as free optimization variables, the resulting error bound does not automatically extend to optimization over neural-network parameters. We next discuss how this error bound transfers from representation space to parameter space. To this end, let $\bm z_i = \phi_{\bm\theta}(\bm x_i) \in \R^d$ denote the feature representation of input $\bm x_i$, where $\phi_{\bm\theta}$ is a neural-network feature map parameterized by $\bm\theta \in \mathbb{R}^p$. For the fixed training set $\bm X = \{\bm x_i\}_{i=1}^N$, define
\[
    Z_X(\bm\theta)
    := [\phi_{\bm\theta}(\bm x_1),\ldots,
        \phi_{\bm\theta}(\bm x_N)]
    \in \mathbb{R}^{d\times N},
    \quad
    G(\bm\theta):=F(Z_X(\bm\theta)),
\]
where $F$ is defined in \eqref{eq:MCR}.  The parameterized problem therefore optimizes $F$ only over the realizable representations $\mathcal R_X:=\{Z_X(\bm\theta):\bm\theta\in\mathbb{R}^p\}$, rather than the entire representation space. Suppose that $Z_X$ is continuously differentiable in the neighborhood under consideration, and write
\[
    \bm J_{\bm\theta}
    :=D(\operatorname{vec}\circ Z_X)(\bm\theta)
    \in\R^{dN\times p},
    \quad
    \bm g(\bm\theta)
    :=\operatorname{vec}\bigl(\nabla F(Z_X(\bm\theta))\bigr).
\]
The chain rule gives $\nabla G(\bm\theta)
    =\bm J_{\bm\theta}^T\bm g(\bm\theta).$ 

To establish a transfer, define $\bm \Theta
    :=\{\bm\theta\in\R^p:Z_X(\bm\theta)\in\mathcal Z\}$, assume that $\bm \Theta\neq\emptyset$, and fix $\bar{\bm \theta} \in \bm \Theta$. Suppose that there exist constants $c_X,\gamma>0$ and a neighborhood $\mathcal U$ of $\bar{\bm\theta}$ such that, for all $\bm\theta\in\mathcal U$,
\begin{align}
    \mathrm{dist}(\bm\theta,\bm \Theta)
    &\le c_X\,
    \mathrm{dist}(Z_X(\bm\theta),\mathcal Z),\quad
    \|\bm J_{\bm\theta}^T\bm g(\bm\theta)\| \ge\gamma\|\bm g(\bm\theta)\|.
    \label{eq:parameter-gradient-control}
\end{align}
Since $Z_X(\bar{\bm\theta})\in\mathcal Z$, continuity allows us to
shrink $\mathcal U$ so that $\mathrm{dist}(Z_X(\bm\theta),\mathcal Z)\le\delta$ throughout $\mathcal U$. Then \Cref{thm:EB}, together with \eqref{eq:parameter-gradient-control}, yields for all $\bm\theta\in\mathcal U$,
\begin{align}
    \mathrm{dist}(\bm\theta,\bm \Theta)
    &\le c_X\,
    \mathrm{dist}(Z_X(\bm\theta),\mathcal Z)
    \le c_X\kappa\|\bm g(\bm\theta)\| \le\frac{c_X\kappa}{\gamma}
    \|\nabla G(\bm\theta)\|.
    \label{eq:parameter-EB}
\end{align}
Thus, $G$ satisfies a local error bound relative to $\bm \Theta$,
with constant $c_X\kappa/\gamma$.

For example, both conditions hold locally if the empirical neural
tangent kernel is uniformly positive definite:
\begin{equation}\label{eq:parameter-NTK}
    \bm J_{\bm\theta}\bm J_{\bm\theta}^T
    \succeq\gamma^2\bm I_{dN},\quad \forall \bm\theta\in\mathcal U.
\end{equation}
This implies the second inequality directly and, by local path lifting, the first inequality in \eqref{eq:parameter-gradient-control} with $c_X=1/\gamma$, after shrinking $\mathcal U$ if necessary.
Thus, \eqref{eq:parameter-EB} holds with constant $\kappa/\gamma^2$. This condition requires $p\ge dN$ and is an additional assumption on the parameterization, not a consequence of \Cref{thm:EB}.
\end{remark}

By leveraging the benign loss landscape of Problem~\eqref{eq:MCR} and the local Lipschitz continuity of its gradient (see \Cref{lem:nabla_lip}), we establish linear convergence of the iterates under the sufficient-ascent and safeguard conditions, without requiring an explicit cost-to-go estimate.

% This can be formalized as follows: 
% \begin{corollary}\label{coro:PL}
% Let $\bm Z^* \in \mathcal{Z}$ and $\bm Z$ be any point satisfying $\|\bm Z - \bm Z^*\|_F \le \delta$.  \\
% (i) (P\L\ inequality) There exists a constant $\eta_1 > 0$ such that 
% \begin{align*}
%     \|\nabla F(\bm Z)\|_F^2 \ge \eta_1\left( F(\bm Z^*) - F(\bm Z) \right). 
% \end{align*}
% (ii) (Quadratic\ growth) There exists a constant $\eta_2 > 0$ such that 
% \begin{align*}
%     \mathrm{dist}(\bm Z, \mathcal{Z}) \le \eta_2\left( F(\bm Z^*) - F(\bm Z) \right).
% \end{align*}
% \end{corollary}
\begin{theorem}\label{coro:rate}
     Suppose that the sequence $\{\bm Z^t\}_{t \ge 0}$ generated for solving Problem \eqref{eq:MCR} satisfies the following properties: \\
    (i) There exists $\mu_1 > 0$ such that 
    \begin{align*}
         F(\bm Z^{t+1}) - F(\bm Z^t) \ge \mu_1\|\bm Z^{t+1} - \bm Z^t\|_F^2. 
    \end{align*}
    % (ii)  There exists $\mu_2 > 0$ for some $\bm Z^* \in \cal Z$ such that 
    % \begin{align*}
    %     F(\bm Z^*) - F(\bm Z^{t+1}) \le \mu_2\left( \mathrm{dist}^2(\bm Z^t, \mathcal{Z}) + \|\bm Z^{t+1} - \bm Z^t\|_F^2\right).
    % \end{align*}
    (ii) There exists $\mu_2 > 0$ such that 
    \begin{align*}
        \|\nabla F(\bm Z^t)\|_F \le \mu_2
        \|\bm Z^{t+1} - \bm Z^t\|_F. 
    \end{align*}
    Suppose in addition that there exists a sufficiently large iteration $t_0 > 0$ such that $\bm{Z}^{t_0} \in \mathcal{U}$, where $\mathcal{U}$ is a neighborhood of $\mathcal{Z}$. Then, for all $t \ge t_0$, we have
    \begin{align}
        &  F(\bm Z^*) - F(\bm Z^{t+1}) \le w^{2}\left(F(\bm Z^*) - F(\bm Z^t)\right), \label{eq:F rate}\\
        & \mathrm{dist}(\bm Z^t,\mathcal{Z}) \le \left( \frac{F(\bm Z^*) - F(\bm Z^{t_0})}{\mu_1} \right)^{1/2} \frac{w^{t - t_0}}{1 - w},\label{eq:Z rate}
    \end{align}
where $\bm{Z}^* \in \mathcal{Z}$, $L_F$ is defined in \eqref{eq:LF}, and $w := \left( 1 + {\mu_1}/\left({2\kappa L_F^2 (\kappa^2 \mu_2^2 + 1)}\right) \right)^{-1/{2}}$. 
\end{theorem}
The proof is deferred to \Cref{pf:rate} and \Cref{subsec:rate}. This theorem shows that the sequence $\{F(\bm{Z}^t)\}_{t \ge t_0}$ converges $Q$-linearly to $F(\bm{Z}^*)$ and $\{\bm{Z}^t\}_{t \ge t_0}$ converges $R$-linearly to the local maximizer set $\mathcal{Z}$. Note that the contraction factor in \Cref{coro:rate} depends on two algorithmic constants, $\mu_1$ and $\mu_2$, and two problem-geometric constants $L_F$ and $\kappa$. A larger sufficient-ascent constant $\mu_1$ improves the rate, whereas larger values of $\mu_2$, $L_F$, or $\kappa$ make the bound more conservative. Since $L_F=\lambda+9\alpha/2=\lambda+9d/(2N\epsilon^2)$ by \Cref{lem:nabla_lip}, the smoothness part of the rate explicitly reflects the coding precision, feature dimension, sample size, and regularization strength. The geometric constant $\kappa$ is characterized in \Cref{prop:eb} through the local singular-value and subspace geometry of the maximizer set.

For gradient ascent, $\bm Z^{t+1}=\bm Z^t+\eta\nabla F(\bm Z^t)$ with $\eta\le 1/L_F$. By the $L_F$-smoothness inequality,
\begin{align*}
F(\bm Z^{t+1})-F(\bm Z^t)
&\ge \eta\|\nabla F(\bm Z^t)\|_F^2
-\frac{L_F\eta^2}{2}\|\nabla F(\bm Z^t)\|_F^2  \\
&=\left(\frac{1}{\eta}-\frac{L_F}{2}\right)
\|\bm Z^{t+1}-\bm Z^t\|_F^2
\ge \frac{L_F}{2}\|\bm Z^{t+1}-\bm Z^t\|_F^2.
\end{align*}
Thus condition (i) holds with $\mu_1=L_F/2$. Moreover, the update gives $\|\nabla F(\bm Z^t)\|_F=\eta^{-1}\|\bm Z^{t+1}-\bm Z^t\|_F$, so condition (ii) holds with $\mu_2=1/\eta$. Combined with the strict-saddle avoidance property of randomly initialized gradient ascent \citep{lee2019first} and the landscape in \Cref{fact:land}, this yields the following corollary. We defer its formal proof to \Cref{subsec:coro GD}. % \pw{Add the explicit rate here!}
% {\color{red} (i) What algorithms satisfy these conditions? (ii) Recall Lemma 3.6! (iii) Strict saddle point + local convergence of gradient descent!}
\begin{corollary}\label{coro:GD}
    Consider gradient ascent with random initialization for solving Problem \eqref{eq:MCR}: 
    \begin{align*}
        \bm Z^{t+1} = \bm Z^t + \eta \nabla F(\bm Z^t),\ t=0,1,2,\dots,
    \end{align*}
    where $\eta > 0$ is a constant step size with $\eta \le 1/{L_{F}}$ and $L_F$ is defined in \eqref{eq:LF}. Then, it holds almost surely that gradient ascent converges to a local maximizer $\bm Z^*$ and there exists a $t_0 > 0$ such that \eqref{eq:F rate} and \eqref{eq:Z rate} hold for all $t \ge t_0$, where $ w=\left( 1 + {1}/\left({4\kappa L_F (\kappa^2/\eta^2 + 1)}\right) \right)^{-1/{2}}.$
\end{corollary}

Finally, \cite{chan2022redunet,wang2024global,yu2024white} constructed a white-box neural network architecture by unrolling iterative optimization algorithms, such as gradient ascent, to optimize the MCR$^2$ and its variants.  These unrolled networks have been shown to be highly effective for learning deep representations. Our characterization in \Cref{thm:EB} sheds light on their effectiveness by revealing favorable geometric properties of the MCR$^2$ landscape. 
 
\vspace{-0.05in}
\section{Proof Overview of Main Results}\label{sec:proof-overview}

In this section, we outline the proof strategy, deferring the technical ingredients to \Cref{sec:proof} and the proofs of the main theorems to \Cref{app:main}. First, we prove the local error-bound property in \Cref{thm:EB} by showing that the gradient norm controls deviations from the local-maximizer geometry, including excess rank, leading singular-value errors, and non-orthogonality between class subspaces. Second, we combine this geometric
error bound with sufficient ascent and a safeguard condition to prove the
local linear convergence result in \Cref{coro:rate}.

\subsection{Proof Sketch of Error Bound}\label{pf:EB}

The main difficulty in proving \Cref{thm:EB} is that the local-maximizer set
$\mathcal Z$ is non-isolated and does not admit a simple projection formula.
Therefore, instead of directly computing $\mathrm{dist}(\bm Z,\mathcal Z)$,
we construct an explicit comparison point in $\mathcal Z$ from the singular
structure of a nearby point $\bm Z$. The proof proceeds in five steps.

\vspace{-0.1in}
\paragraph{Step 1: Singular value decomposition.} Fix $\bm Z^*\in\mathcal Z$, and let $r_k:=\mathrm{rank}(\bm Z_k^*)$ and $r:=\sum_{k=1}^K r_k$. For $\bm Z=[\bm Z_1\ \cdots\ \bm Z_K]$ sufficiently close to $\bm Z^*$, consider the SVD
\begin{align}\label{eq:SVD Zk}
    \bm Z_k = \bm P_k \bm \Sigma_k \bm Q_k^T  = \begin{bmatrix}
        \bm P_{k}^{(1)} & \bm P_{k}^{(2)}
    \end{bmatrix}\begin{bmatrix}
        \bm \Sigma_k^{(1)} & \bm 0 \\
        \bm 0 & \bm \Sigma_k^{(2)} 
    \end{bmatrix}\begin{bmatrix}
        \bm Q_{k}^{(1)^T} \\
         \bm Q_{k}^{(2)^T}
    \end{bmatrix},\ \forall k \in [K]. 
\end{align}
Here, $\bm P_k \in \mathcal{O}^d$, $\bm P_{k}^{(1)} \in \mathcal{O}^{d\times r_k}$, and $\bm P_{k}^{(2)} \in \mathcal{O}^{d\times (d-r_k)}$; $\bm \Sigma_k \in \R^{d\times N_k}$, $\bm \Sigma_k^{(1)} = \mathrm{diag}(\sigma_{k,1},\dots,\sigma_{k,r_k})$ with $\sigma_{k,1} \ge \dots \ge \sigma_{k,r_k}$, and $\bm \Sigma_k^{(2)} \in \R^{(d-r_k)\times (N_k-r_k)}$; $\bm Q_k \in \mathcal{O}^{N_k}$, $\bm Q_{k}^{(1)} \in \mathcal{O}^{N_k \times r_k}$, and $\bm Q_{k}^{(2)} \in \mathcal{O}^{N_k \times (N_k-r_k)}$. We also define the truncated and compressed representations $\bm Z_k^{(1)}:=\bm P_k^{(1)}\bm\Sigma_k^{(1)}{\bm Q_k^{(1)}}^T$, $\bm Z^{(1)} :=[\bm Z_1^{(1)}\ \cdots\ \bm Z_K^{(1)}] \in\R^{d\times N}$, $\bar{\bm Z}_k:=\bm P_k^{(1)}\bm\Sigma_k^{(1)}$, and $\bar{\bm Z}:=[\bar{\bm Z}_1\ \cdots\ \bar{\bm Z}_K] \in \R^{d\times r}$.

\vspace{-0.1in}
\paragraph{Step 2: Reduction to three aggregate estimates.} We reduce the desired error bound to three geometric estimates by
introducing a comparison point in $\mathcal Z$. For any collection of matrices $\hat{\bm U}_k\in\mathcal O^{d\times r_k}$ satisfying
$\hat{\bm U}_k^T\hat{\bm U}_l=\bm 0$ for all $k\ne l$, define $\hat{\bm Z}_k :=\sigma_k^*\hat{\bm U}_k{\bm Q_k^{(1)}}^T$ and 
$\hat{\bm Z} :=[\hat{\bm Z}_1\ \cdots\ \hat{\bm Z}_K].$
By \Cref{fact:land}(ii), we have $\hat{\bm Z}\in\mathcal Z$.
Using the SVD in \eqref{eq:SVD Zk}, the triangle inequality,
and the orthogonal invariance of the Frobenius norm, we obtain
\begin{align}\label{eq:proof-sketch-decomp}
    \mathrm{dist}(\bm Z,\mathcal Z)
    &\le \|\bm Z-\hat{\bm Z}\|_F
    \le \sum_{k=1}^K\|\bm Z_k-\hat{\bm Z}_k\|_F
    \notag\\
    &\le \sum_{k=1}^K\Big(
    \underbrace{
        \|\bm\Sigma_k^{(1)}-\sigma_k^*\bm I_{r_k}\|_F
    }_{A_k}
    +\underbrace{
        \sigma_k^*\|\bm P_k^{(1)}-\hat{\bm U}_k\|_F
    }_{B_k}
    +\underbrace{
        \|\bm\Sigma_k^{(2)}\|_F
    }_{C_k}
    \Big).
\end{align}
It therefore remains to construct $\{\hat{\bm U}_k\}_{k=1}^K$ satisfying the above orthogonality conditions and to bound each of the aggregate quantities $\sum_{k=1}^K A_k$, $\sum_{k=1}^K B_k$, and $\sum_{k=1}^K C_k$ by a constant multiple of $\|\nabla F(\bm Z)\|_F$. 

\vspace{-0.1in}
\paragraph{Step 3: Controlling the tail singular values.}
We first control the tail singular values, which vanish at
$\bm Z^*$. By \Cref{lem:bound sing}, we have for $\bm Z$ sufficiently close to $\bm Z^*$,
\[
    \sigma_i(\bm Z_k) \le \frac{2}{\lambda}\|\nabla F(\bm Z)\|_F, \quad \forall r_k<i\le\min\{d,N_k\},\ k\in[K].
\]
Consequently, the tail contribution of each class satisfies
\[
    C_k
    =\|\bm\Sigma_k^{(2)}\|_F
    =\left(
        \sum_{i=r_k+1}^{\min\{d,N_k\}}\sigma_i(\bm Z_k)^2
      \right)^{1/2}
    \le
    \frac{2\sqrt{\min\{d,N_k\}-r_k}}{\lambda}
    \|\nabla F(\bm Z)\|_F.
\]
Summing over classes yields
\[
    \sum_{k=1}^K C_k
    \le c_C\|\nabla F(\bm Z)\|_F,
    \quad
    c_C:=
    \frac{2}{\lambda}
    \sum_{k=1}^K\sqrt{\min\{d,N_k\}-r_k}
    \le \frac{2K\sqrt d}{\lambda}.
\]
% \paragraph{Step 3: Controlling the tail singular values.} We first show that tail singular directions cannot persist near a local maximizer unless the gradient is large. More precisely, for every $i>r_k$, there exists $c_1 > 0$ such that
% \[
%     \sigma_i(\bm Z_k)
%     \le  c_1\|\nabla F(\bm Z)\|_F,
% \]
% as formalized in \Cref{lem:bound sing}. Thus, after summing over classes and absorbing dimension-dependent factors into the constant, there exists a positive constant $c_C$ such that we can bound the sum of the $C$-part as $\sum_{k=1}^K C_k=\sum_{k=1}^K\|\bm\Sigma_k^{(2)}\|_F\le c_C\|\nabla F(\bm Z)\|_F$.

\vspace{-0.1in}
\paragraph{Step 4: Controlling the leading singular values and
cross-subspace products.}
For $\bm Z$ sufficiently close to $\bm Z^*$, the tail estimate and
\eqref{eq7:prop sing} give
$\|\nabla F(\bm Z^{(1)})\|_F
\le c_T\|\nabla F(\bm Z)\|_F$ for some $c_T>0$.
Using \Cref{lem:inverse_identity}, we express the truncated gradient
in terms of $\bar{\bm Z}^T\bar{\bm Z}$ and bound the associated
matrix residual by the gradient norm.
The off-diagonal blocks control the cross-subspace products
${\bm P_k^{(1)}}^T\bm P_l^{(1)}$ via \Cref{lem:spectral},
while the diagonal blocks control $|\varphi_k(\sigma_{k,i})|$, where
\[
    \varphi_k(x)
    :=\frac{\alpha}{1+\alpha x^2}
      -\frac{\alpha}{1+\alpha_kx^2}-\lambda,
    \qquad \varphi_k(\sigma_k^*)=0.
\]
The local separation estimate
$|\varphi_k(x)|\ge\kappa_3|x-\sigma_k^*|$
for $|x-\sigma_k^*|\le\delta_1$ in \Cref{lem:g}
then bounds the leading singular-value errors.
Combining these estimates as in the proof of
\Cref{prop:bound sing} yields
\begin{align}\label{eq:proof-sketch-intermediate}
    \sum_{k=1}^K
    \|\bm\Sigma_k^{(1)}-\sigma_k^*\bm I_{r_k}\|_F
    +\sum_{k=1}^K\sum_{l\ne k}
    \|{\bm P_k^{(1)}}^T\bm P_l^{(1)}\|_F
    \le c_2\|\nabla F(\bm Z)\|_F
\end{align}
for a local constant $c_2>0$. This bounds $\sum_k A_k$ and provides the cross-subspace control needed for Step 5.

\vspace{-0.1in}
\paragraph{Step 5: Orthogonalization and construction of a nearby maximizer.}
We use the blockwise orthogonalization in \Cref{lem:sing matrix}.
Set $\hat{\bm U}_1:=\bm P_1^{(1)}$ and, for $k\ge2$, define
\[
    \bm H_k
    :=\left(\bm I_d-
      \sum_{l=1}^{k-1}\hat{\bm U}_l\hat{\bm U}_l^T\right)
      \bm P_k^{(1)}.
\]
Given a thin SVD
$\bm H_k=\tilde{\bm U}_k\tilde{\bm\Sigma}_k\tilde{\bm V}_k^T$,
set $\hat{\bm U}_k:=\tilde{\bm U}_k\tilde{\bm V}_k^T$,
omitting empty blocks.
For $\bm Z$ sufficiently close to $\bm Z^*$,
\Cref{lem:sing matrix} ensures that the resulting matrices
$\hat{\bm U}_k\in\mathcal O^{d\times r_k}$ have mutually
orthogonal column spaces and satisfy
\[
    \|\bm P_k^{(1)}-\hat{\bm U}_k\|_F
    \le c_3\|\nabla F(\bm Z)\|_F,
    \quad \forall k\in[K],
\]
for a local constant $c_3>0$. Consequently, we have 
\[
    \sum_{k=1}^K B_k
    \le c_3\sum_{k=1}^K\sigma_k^* \|\nabla F(\bm Z)\|_F.
\]
Combining this estimate with Steps 3--4 in
\eqref{eq:proof-sketch-decomp} gives a local error bound around
$\bm Z^*$, as formalized in \Cref{prop:eb}.
Finally, a finite-covering argument over the compact set
$\mathcal Z$ yields uniform constants $\delta,\kappa>0$,
completing the proof of \Cref{thm:EB}; see \Cref{app:thm EB}.

% \subsection{Proof Sketch of Convergence Rate}

% For the rate result, the proof uses the error bound only after the iterates enter a Lyapunov-stable neighborhood of one connected component of $\mathcal{Z}$. On such a neighborhood, the error bound and smoothness imply a local cost-to-go estimate of the form
% \[
% F(\bm Z^*)-F(\bm Z)\le c_{\rm gap}\|\nabla F(\bm Z)\|_F^2.
% \]
% Together with the safeguard condition and Lipschitz continuity of the gradient, this yields
% \[
% F(\bm Z^*)-F(\bm Z^{t+1})
% \le c_{\rm step}\|\bm Z^{t+1}-\bm Z^t\|_F^2.
% \]
% Combining this inequality with sufficient ascent gives a uniform contraction of the objective gap, which is precisely \eqref{eq:F rate}; summing the contracted step lengths then gives the iterate-distance bound \eqref{eq:Z rate}. This completes the proof idea for \Cref{coro:rate} and \Cref{coro:GD}.
\subsection{Proof Sketch of Convergence Rate}\label{pf:rate}

The previous subsection provides the key geometric ingredient for the convergence analysis: near the local-maximizer set $\mathcal Z$, the distance to $\mathcal Z$ is controlled by the gradient norm. We use this error bound after the iterates enter a neighborhood of one connected component $\mathfrak X$ of $\mathcal Z$. This localization is necessary because $\mathcal Z$ is generally non-isolated, while different connected components are separated by the landscape characterization in \Cref{fact:land}. Once the iterates enter a sufficiently small neighborhood of $\mathfrak X$, the sufficient-ascent condition ensures monotonicity of the objective, and a standard Lyapunov-stability argument keeps the trajectory inside the same basin.

For $t\ge t_0$, let $e_t:=F^*-F(\bm Z^t)\ge0$ and
$s_t:=\|\bm Z^{t+1}-\bm Z^t\|_F$.
Within this neighborhood, the local P\L\ estimate
\eqref{eq:plinequality}, smoothness, and the error-bound and
safeguard conditions yield the cost-to-go estimate
\[
\begin{aligned}
    e_{t+1}
    &\le \kappa\|\nabla F(\bm Z^{t+1})\|_F^2
    \le \kappa L_F^2
    \mathrm{dist}^2(\bm Z^{t+1},\mathfrak X)\\
    &\le 2\kappa L_F^2
    \bigl(\mathrm{dist}^2(\bm Z^t,\mathfrak X)+s_t^2\bigr) \le 2\kappa L_F^2(\kappa^2\mu_2^2+1)s_t^2.
\end{aligned}
\]
Combining this with sufficient ascent,
$e_t-e_{t+1}\ge\mu_1s_t^2$, gives
\[
    e_{t+1}\le w^2e_t,
    \quad
    w:=\left(
    1+\frac{\mu_1}{2\kappa L_F^2(\kappa^2\mu_2^2+1)}
    \right)^{-1/2}\in(0,1),
\]
which proves the $Q$-linear objective-gap decay in
\eqref{eq:F rate}.
Moreover, sufficient ascent gives
$s_t\le\sqrt{e_{t_0}/\mu_1}\,w^{t-t_0}$.
These step lengths are summable, so the iterates converge.
The bound
$\mathrm{dist}(\bm Z^t,\mathfrak X)\le\kappa\mu_2s_t\to0$
ensures that their limit belongs to $\mathfrak X$.
Summing the geometric tail gives the $R$-linear bound in
\eqref{eq:Z rate}, completing the proof sketch of \Cref{coro:rate}.

For fixed-step gradient ascent with $0<\eta\le1/L_F$, smoothness
and the update identity verify the sufficient-ascent and safeguard
conditions with $\mu_1=L_F/2$ and $\mu_2=1/\eta$, respectively.
Substituting these constants gives the local rates in
\Cref{coro:GD} after sufficiently late entry into $\mathcal U$.

\section{Experimental Results}\label{sec:expe}

In this section, we conduct numerical experiments on synthetic datasets in \Cref{subsec:exp1} and real-world datasets in \Cref{subsec:exp2} in support of our theoretical findings. Since the main contribution of this paper is a convergence-rate theory for the MCR$^2$ objective, the primary quantities reported here are the objective gap and the iterate gap against the number of iterations. We also provide additional experiments in \Cref{app:extension-experiments} to show how optimizing the MCR$^2$ objective relates to feature geometry and linear-probe performance.

\begin{figure*}[t]
    \centering
    \begin{subfigure}[b]{0.48\textwidth}
    \includegraphics[width=\textwidth]{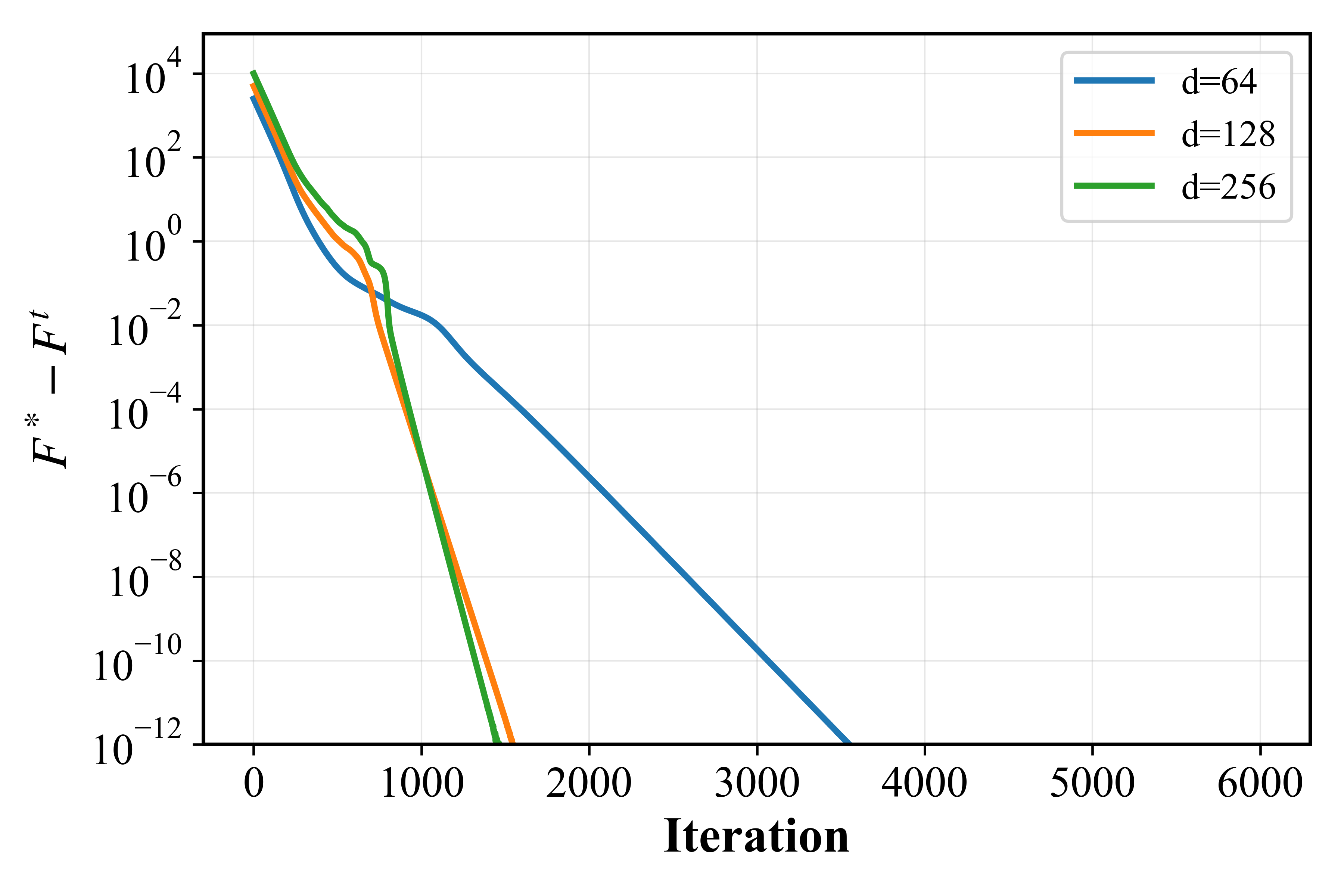}\vspace{-0.1in} 
        % \caption{Function values}
        \label{fig:image1}
    \end{subfigure}
    \hfill
    \begin{subfigure}[b]{0.48\textwidth}
        \includegraphics[width=\textwidth]{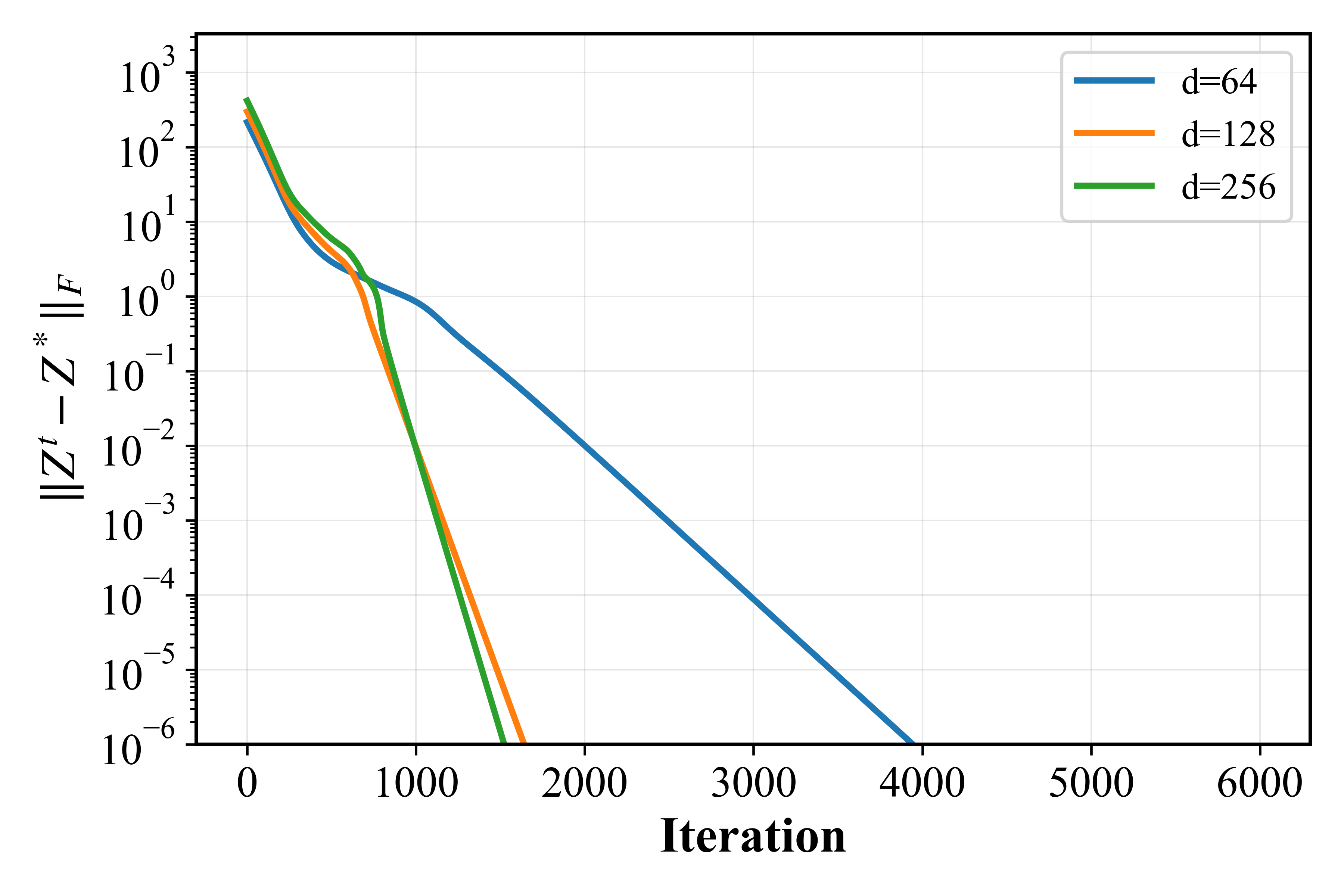}\vspace{-0.1in}
        % \caption{Iterate distance}
        \label{fig:image2}
    \end{subfigure}\vspace{-0.1in}
    \caption{{\bf Visualization of the convergence rate of gradient ascent for solving Problem \eqref{eq:MCR}.} In both figures, the $x$-axis is the number of iterations, denoted by $t$. (Left) We plot the function value gap $F^* - F^t$ against number of iterations, where $F^t = F(\bm Z^t)$ and $F^* = F(\bm Z^*)$. (Right) We plot the iterate gap $\|\bm Z^t - \bm Z^*\|_F$ against number of iterations, where $\bm Z^*$ denotes the last iterate.} \vspace{-0.05in}
    \label{fig:all_images} 
\end{figure*}

\subsection[Convergence Rate of Gradient Ascent for MCR2]{Convergence Rate of Gradient Ascent for MCR$^2$}\label{subsec:exp1}

In this subsection, we apply gradient ascent (GA) for solving Problem \eqref{eq:MCR}. The problem settings are as follows: We set the number of classes $K=4$ and the number of data samples $N_k=200$ for each $k \in [K]$, and vary the data dimension $d \in \{ 64,128,256\}$. Moreover, we set the regularization parameter $\lambda=0.1$ and the coding precision $\epsilon=0.5$. For the implementation of GA, the iterate is initialized using a standard Gaussian random distribution and adopts a constant step size of  0.1.

To evaluate the convergence rate of GA in terms of function values (resp., iterate gap), we plot $F(\bm Z^*)-F(\bm Z^t)$ (resp., $\|\bm Z^t - \bm Z^*\|_F$) against the number of iterations. Here, $\bm Z^t$ denotes the $t$-th iterate  and $\bm Z^*$ denotes the last iterate generated by GA, which serves as a surrogate for a local maximizer.\footnote{According to \cite{lee2019first, wang2024global}, gradient ascent with random initialization converges to a local maximizer almost surely.} It is observed that GA with random initialization converges to a local maximizer at a linear rate in terms of both the function value gap and the iterate distance. This is consistent with the linear convergence guarantee established in \Cref{coro:rate}, which in turn provides further numerical evidence supporting the error-bound condition in \Cref{thm:EB}. This controlled synthetic experiment is designed to isolate the local linear regime predicted by the theory; the resulting trajectories are shown in \Cref{fig:all_images}.

\begin{figure*}[t]
    \centering

    \begin{subfigure}[b]{0.45\textwidth}
        \includegraphics[width=\textwidth]{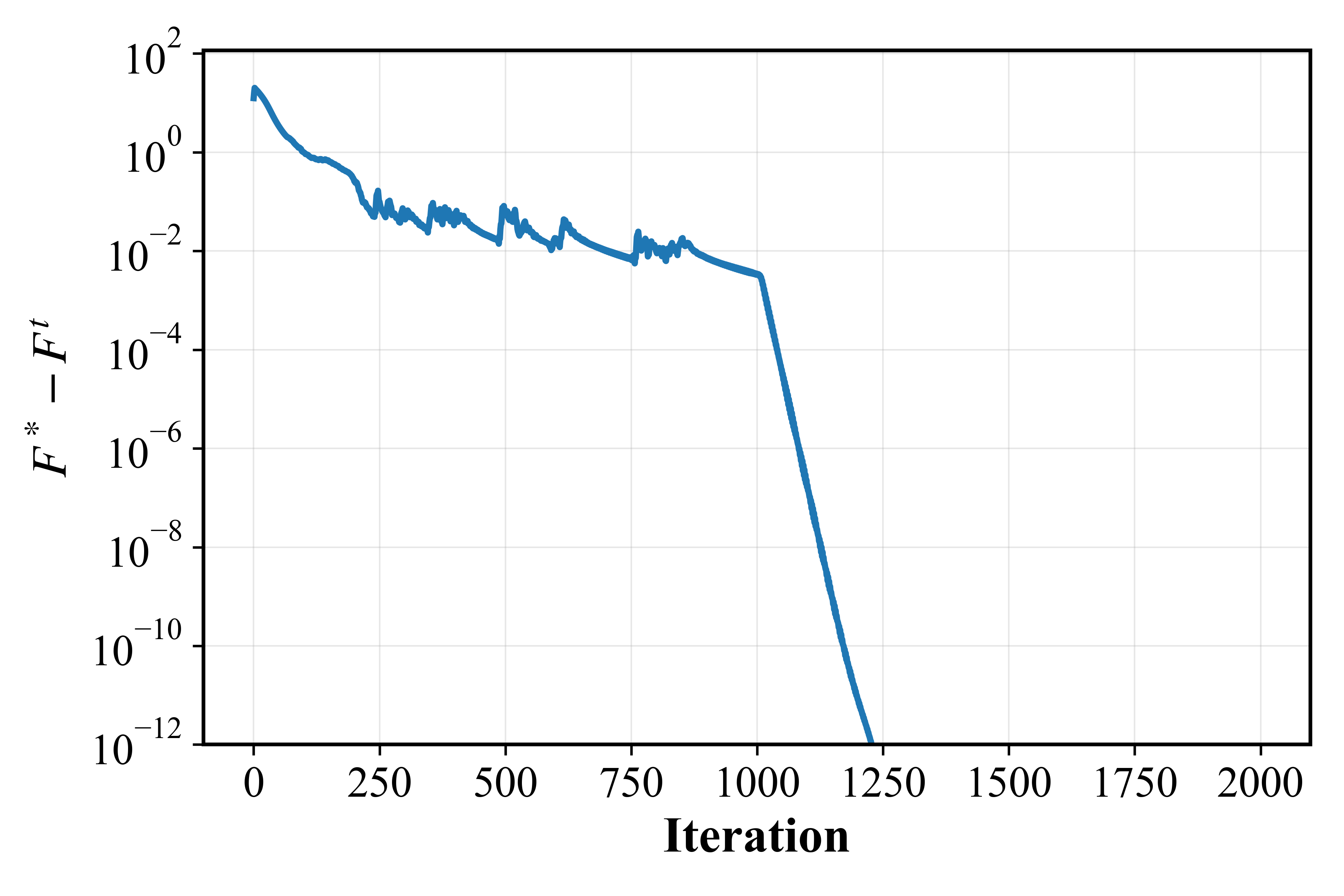}
        \label{fig:mnist_obj_gap}
    \end{subfigure}
    \hfill
    \begin{subfigure}[b]{0.45\textwidth}
        \includegraphics[width=\textwidth]{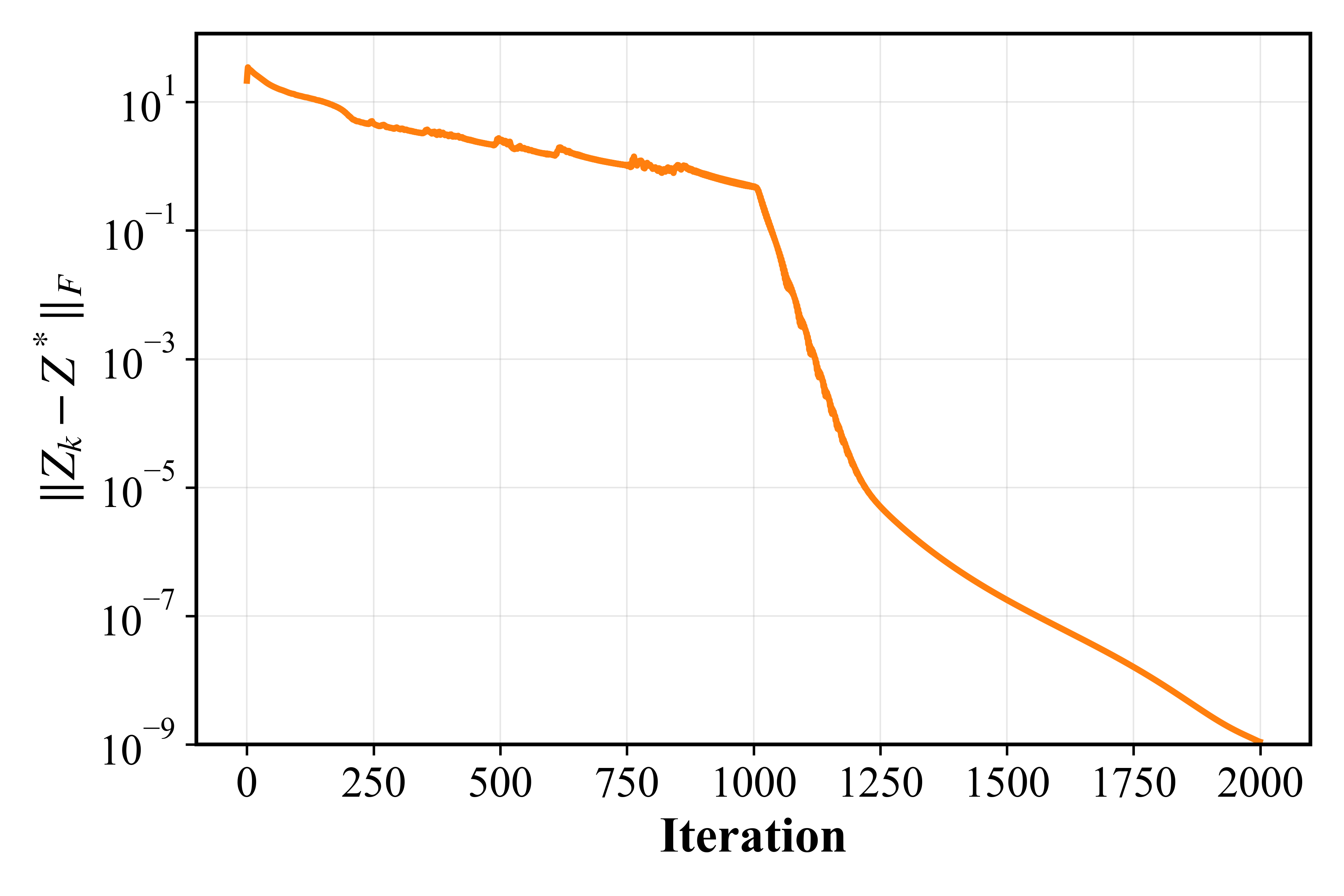}
        \label{fig:mnist_iter_gap}
    \end{subfigure}

    \vspace{-0.1in}

    \begin{subfigure}[b]{0.45\textwidth}
        \includegraphics[width=\textwidth]{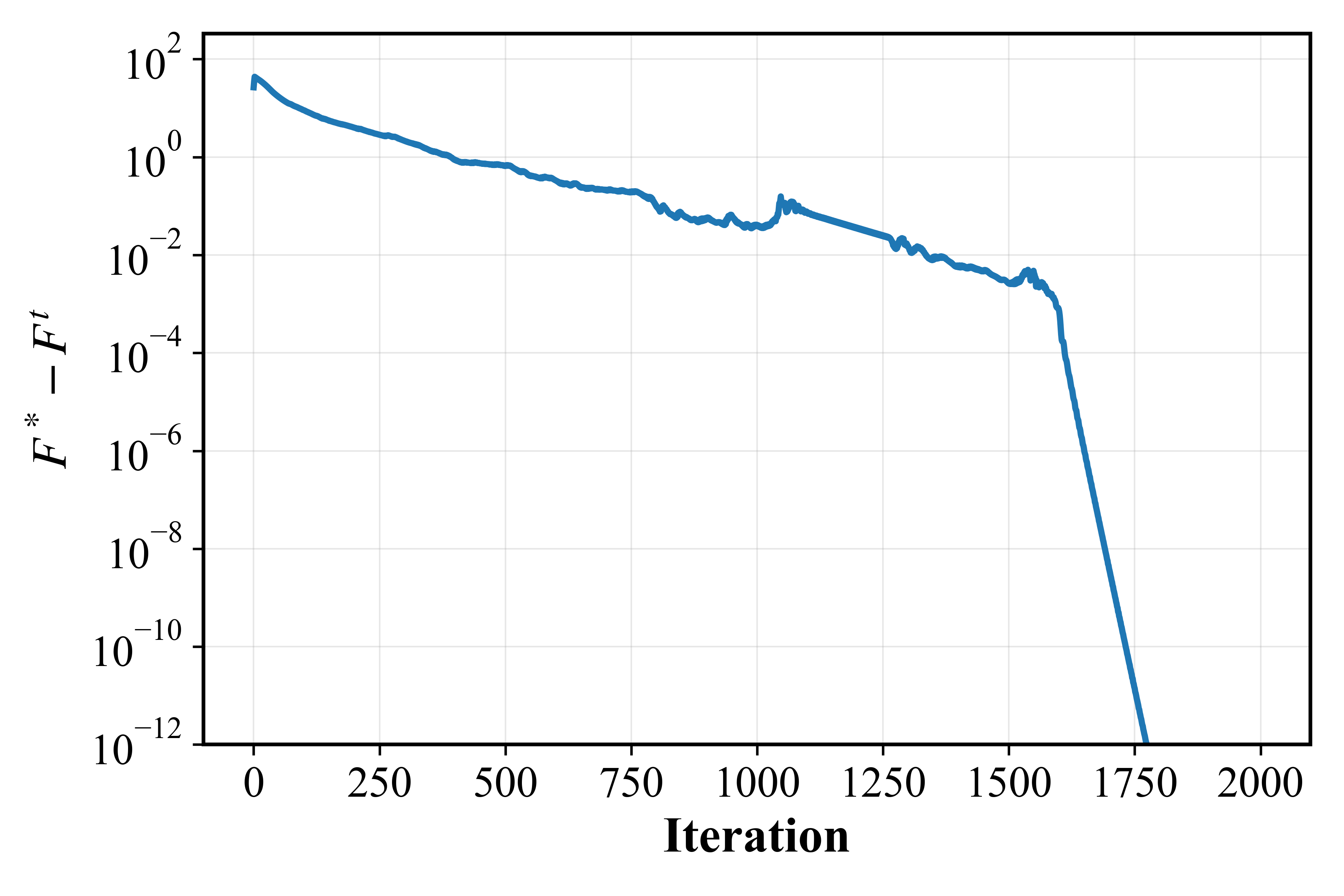}
        \label{fig:cifar_obj_gap}
    \end{subfigure}
    \hfill
    \begin{subfigure}[b]{0.45\textwidth}
        \includegraphics[width=\textwidth]{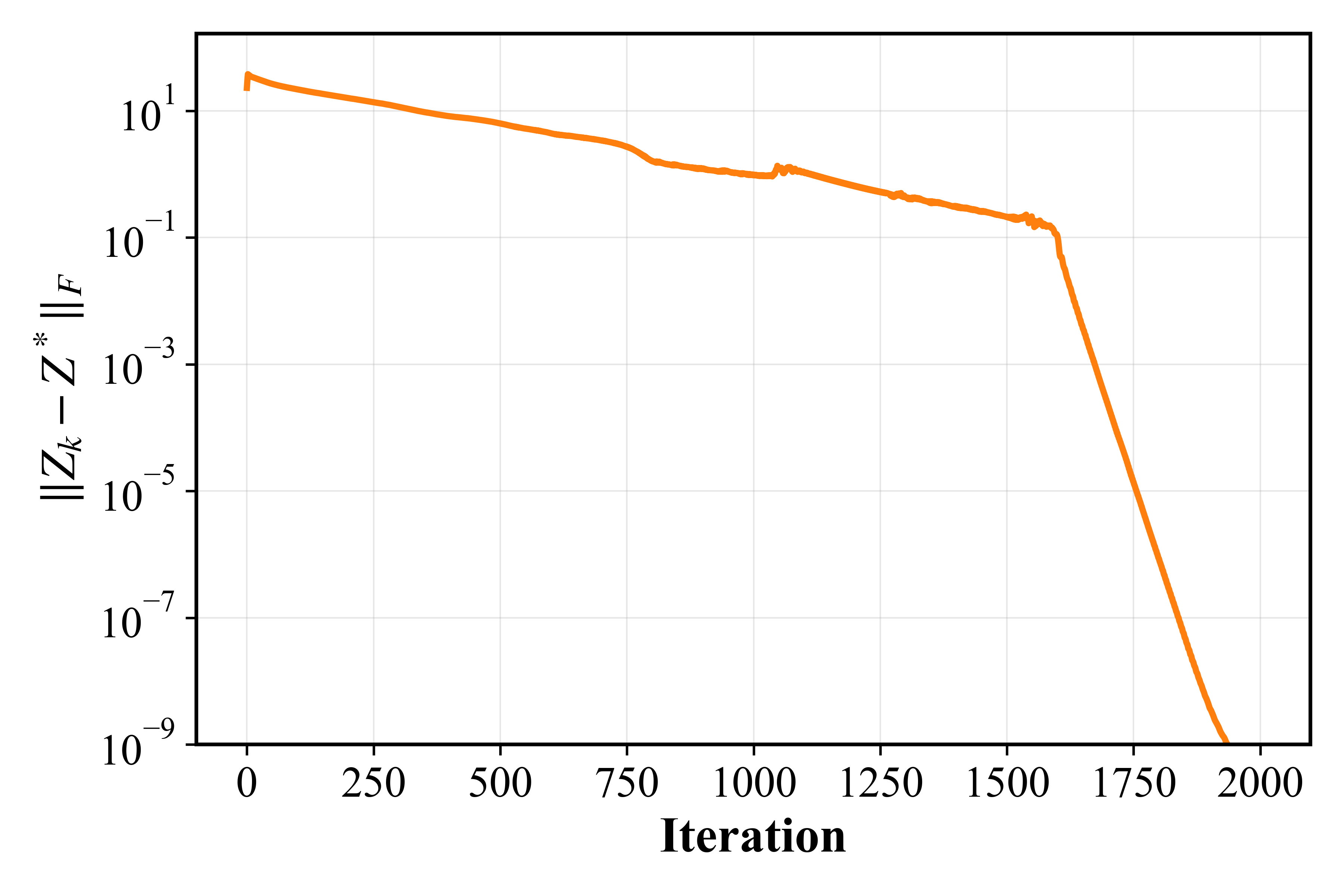}
        \label{fig:cifar_iter_gap}
    \end{subfigure}
    \vspace{-0.15in}
    \caption{{\bf Visualization of the convergence rate on MNIST and CIFAR-10 using an MLP.}
    We visualize the objective value gap and the iterate distance against the number of iterations. Top row: MNIST. Bottom row: CIFAR-10. Left column: objective value gap. Right column: iterate distance.}\vspace{-0.2in}
    \label{fig:real_convergence} 
\end{figure*}

\subsection[Convergence Rate of Training Deep Networks via MCR2]{Convergence Rate of Training Deep Networks via MCR$^2$}\label{subsec:exp2}

In this subsection, we extend our empirical studies to train neural networks for classification problems on real-world datasets, namely MNIST \cite{lecun1998gradient} and CIFAR-10 \cite{krizhevsky2009learning}. Let $\{\bm x_i\}_{i=1}^N \subseteq \R^D$ denote the data samples. We employ a neural network $f_{\bm \theta}(\cdot):\R^D \to \R^d$ to learn data representations $\bm z_i = f_{\bm \theta}(\bm x_i)$ for each $i \in [N]$. The network parameters $\bm \theta$ are trained by optimizing Problem \eqref{eq:MCR}. 

In our implementations, we employ an over-parameterized multi-layer perceptron (MLP) network with residual connections to parameterize $f_{\bm \theta}$. 
We adopt a three-layer MLP with ReLU activations and optimize the MCR$^2$ objective using the AdamW optimizer equipped with a cosine annealing learning rate scheduler \cite{loshchilov2018decoupled}. 
To ensure optimization stability, the schedule incorporates a linear warmup for the initial 5\% of iterations, followed by a cosine decay towards a minimum learning rate of $10^{-6}$.
The experiments are conducted with the coding precision $\epsilon=0.4$ and a regularization parameter $\lambda=10^{-4}$. 
The initial learning rate is set to $3.5 \times 10^{-4}$ for MNIST and $5 \times 10^{-4}$ for CIFAR-10.
All optimization runs are carried out in a deterministic full-batch setting to rigorously evaluate convergence behavior.

We plot the convergence behavior on MNIST and CIFAR-10 in \Cref{fig:real_convergence}. On both datasets, the optimization exhibits a two-phase pattern: an initial exploration phase over the non-convex landscape, followed by a sharp transition into a local neighborhood of a local maximizer. After this transition, both the objective gap and the iterate distance decay, consistent with the predicted local linear convergence. The transition occurs around iteration $1{,}000$ on MNIST and around iteration $1{,}600$ on CIFAR-10, reflecting the greater complexity of CIFAR-10. Additional experiments in \Cref{app:extension-experiments} further examine whether optimization progress remains aligned with representation quality under minibatch training, using linear-probe accuracy and class-mean cosine similarity.

\vspace{-0.1in}
\section{Conclusions}\label{sec:con}

In this work, we studied the optimization geometry of the MCR$^2$ problem by showing that it satisfies an error-bound property in a neighborhood of its local maximizers. This property directly implies key regularity conditions, including the P\L\ inequality and quadratic growth, and allows us to establish that first-order methods converge linearly to a local maximizer.  Our experimental results on synthetic and real-world datasets strongly supported our theoretical results. One future direction is to extend our analysis of Problem \eqref{eq:MCR} to the constrained MCR$^2$ problem studied in \cite{yu2020learning}, which imposes a norm constraint on the features instead of using a regularizer, as well as to the sparse MCR$^2$ problem in \cite{yu2024white}, which considers a sparsity-inducing regularization on the features. Another promising direction is to theoretically analyze Problem \eqref{eq:MCR} when the feature representations are parameterized by deep neural networks, as explored empirically in \Cref{fig:cifar_iter_gap}. Such an analysis could provide insights into the interplay between the network architecture and the optimization landscape.

\section*{Acknowledgement}
R. Jiang is supported by the Fundamental and Interdisciplinary Disciplines Breakthrough Plan of the Ministry of Education of China under grant JYB2025XDXM904 and the National Natural Science Foundation of China under grants 72394360 and 72394364. 
P. Wang is supported in part by the National Natural Science Foundation of China under grant 12601585 and in part by the Macau Science and Technology Development Fund (FDCT) under grant 0091/2025/ITP2. H. Liu is supported by the National Natural Science Foundation of China under grants 12301403, 72192830, and 72192832. 

% impact\section*{Impact Statement}
% This paper presents work whose goal is to advance the field of machine learning. There are many potential societal consequences of our work, none of which we feel must be specifically highlighted here.
\bibliographystyle{abbrvnat}
\bibliography{NeurIPS/MCR_reference}

%%%%%%%%%%%%%%%%%%%%%%%%%%%%%%%%%%%%%%%%%%%%%%%%%%%%%%%%%%%%%%%%%%%%%%%%%%%%%%%
%%%%%%%%%%%%%%%%%%%%%%%%%%%%%%%%%%%%%%%%%%%%%%%%%%%%%%%%%%%%%%%%%%%%%%%%%%%%%%%
% APPENDIX
%%%%%%%%%%%%%%%%%%%%%%%%%%%%%%%%%%%%%%%%%%%%%%%%%%%%%%%%%%%%%%%%%%%%%%%%%%%%%%%
%%%%%%%%%%%%%%%%%%%%%%%%%%%%%%%%%%%%%%%%%%%%%%%%%%%%%%%%%%%%%%%%%%%%%%%%%%%%%%%
\newpage
\appendix
\onecolumn
\begin{center}
{\Large \bf Appendices}
\end{center}\vspace{-0.15in}
\par\noindent\rule{\textwidth}{1pt}


In the supplementary, we provide proofs of the technical results presented in \Cref{sec:main} and \Cref{sec:proof-overview}. Besides the notions introduced earlier, we shall use $\mathrm{BlkD}(\bm X_1,\dots,\bm X_K)$ to denote the block diagonal matrix whose diagonal blocks are $\bm X_1,\dots,\bm X_K$. % The organization is as follows: In  \Cref{sec:grad}, we introduce some preliminary setups and auxiliary results for studying the MCR$^2$ problem. Then, we proceed to prove the technical results concerning the global optimality of Problem \eqref{eq:MCR}  in \Cref{sec:pf opti} and the optimization landscape of Problem \eqref{eq:MCR} in  \Cref{sec:pf land}, respectively. Finally, we prove the main theorems in \Cref{thm:1}  and \Cref{thm:2} in \Cref{sec:pf main}. 

\section{Technical Ingredients for the Error Bound}\label{app:eb-ingredients}\label{sec:proof}

In this section, we present the proofs of the main theorem in \Cref{sec:main}. Recall that $N$ denotes the number of training samples and $N_k$ denotes the number of training samples in the $k$-th class for each $k \in [K]$. For ease of exposition, we define
\begin{align}\label{eq:alp}
\alpha := \frac{d}{N\epsilon^2},\ \alpha_k := \frac{d}{N_k\epsilon^2},\ \forall k \in [K].
\end{align} 
For each $k \in [K]$, we compute 
\begin{align}\label{eq:grad Fk}
    \nabla_k F(\bm Z) := &\ \frac{\partial F(\bm Z)}{\partial \bm Z_k} = \alpha \left(\bm I_d + \alpha\bm Z\bm Z^T \right)^{-1}\bm Z_k - \alpha \left( \bm I_d + \alpha_k\bm Z_k \bm Z_k^T \right)^{-1}\bm Z_k - \lambda\bm Z_k. 
\end{align} 

\subsection{Preliminary Results}\label{subsec:pre} 

We use the characterization of the local maximizer set $\mathcal{Z}$ stated in \Cref{fact:land}(ii). 
For ease of exposition, given $\bm Z^* \in \mathcal{Z}$ with rank $r_k$ and singular values $\sigma_1^*,\dots,\sigma_{r_k}^*$, let 
\begin{align*}
\sigma_{\max}^* := \max\left\{ \sigma_i^*: i \in [r_k] \right\},\ \sigma_{\min}^* := \min\left\{ \sigma_i^*: i \in [r_k] \right\}. 
\end{align*}
Now, we prove the following inequalities to bound the singular values for any $\bm Z$ in a neighborhood of a local maximizer $\bm Z^* \in \mathcal{Z}$. 
\begin{lemma}\label{lem:spectral}
    Let $\bm Z^* \in \mathcal{Z}$ be arbitrary and $r_k = \mathrm{rank}(\bm Z_k^*)$ for each $k \in [K]$.
    % Let $\bm r \in \cal R$ and 
    For any $\bm Z = [\bm Z_1\ \cdots\ \bm Z_K]$ satisfying 
    \begin{align}\label{eq:dist 1}
        \|\bm Z - \bm Z^*\|_F \le \frac{\sigma^*_{\rm min}}{2},
    \end{align}
we have for each $k \in [K]$,
\begin{align}
  & \frac{\sigma^*_{\rm min}}{2} \le  \sigma_i(\bm Z_k) \le \frac{3\sigma^*_{\rm max}}{2},\ \forall i \in [r_k], \label{eq:spec1}\\
  & \left|\sigma_i(\bm Z_k)\right| \le \frac{\sigma_{\rm min}^*}{2},\ i=r_k+1,\dots,\min\{d,N_k\}. \label{eq:spec2}
\end{align}
\end{lemma}
\begin{proof}[Proof of \Cref{lem:spectral}] 
    %Let $\bm Z^* = [\bm Z_1^*\ \cdots\ \bm Z_K^*] \in \mathcal{Z}_{\bm r}$ be such that $\mathrm{dist}(\bm Z, \mathcal{Z}_{\bm r}) = \|\bm Z - \bm Z^*\|_F$. 
    Using Weyl's inequality (see, e.g., \cite[Corollary 7.3.5]{horn2012matrix}), we have for each $k \in [K]$ and $i \in \{1,\dots,\min\{d,N_k\}\}$,
    \begin{align*}
        \left| \sigma_i(\bm Z_k) - \sigma_i(\bm Z_k^*) \right| \le \left\| \bm Z_k - \bm Z_k^* \right\| \le \left\| \bm Z_k - \bm Z_k^* \right\|_F \le \|\bm Z - \bm Z^*\|_F \le \frac{\sigma^*_{\rm min}}{2}.
    \end{align*}
    This, together with $\sigma^*_{\min} \le \sigma_i(\bm Z_k^*) \le \sigma^*_{\max}$ for each $i \in [r_k]$ and $\sigma_i(\bm Z_k^*) = 0$ for all $i \in \{r_k+1,\dots,\min\{d,N_k\}\}$, directly implies \eqref{eq:spec1} and \eqref{eq:spec2}.   
\end{proof}

Based on the above result, we are now ready to show that the last $N_k-r_k$ singular values of $\bm Z$ can be bounded by its gradient norm. 
\begin{lemma}\label{lem:bound sing}
    Let $\bm Z^* \in \mathcal{Z}$ be arbitrary and $r_k = \mathrm{rank}(\bm Z_k^*)$ for each $k \in [K]$. For any $\bm Z = [\bm Z_1\ \cdots\ \bm Z_K]$ satisfying 
    \begin{align}\label{eq:dist 2}
        \|\bm Z - \bm Z^*\|_F \le \sqrt{\frac{\lambda}{2\alpha\alpha_{\max}}},
    \end{align}
      it holds for all $k \in [K]$ and $i \in \{ r_k+1,\dots,\min\{d,N_k\}\}$ that 
    \begin{align}\label{eq:eb sing 1}
        \sigma_i(\bm Z_k) \le  \frac{2\|\nabla F(\bm Z)\|_F}{\lambda}.
    \end{align}
\end{lemma} 
\begin{proof}[Proof of \Cref{lem:bound sing}]
For ease of exposition, we denote $\sigma_{k,r_k+1} := \sigma_{r_k+1}(\bm Z_k)$ for each $k \in [K]$. Using Weyl's inequality and $\sigma_{r_k+1}(\bm Z_k^*) = 0$, we obtain 
\begin{align}\label{eq1:lem bound sing}
    \left| \sigma_{k,r_k+1} \right| \le \|\bm Z_k - \bm Z_k^*\| \le  \|\bm Z - \bm Z^*\|_F \overset{\eqref{eq:dist 2}}{\le} \sqrt{\frac{\lambda}{2\alpha\alpha_{\max}}}.
\end{align}
Plugging the SVD of $\bm Z_k$ in \eqref{eq:SVD Zk} into $\nabla_k F(\bm Z)$ in \eqref{eq:grad Fk} yields for each $k \in [K]$, 
\begin{align}
    \nabla_k F(\bm Z)
    ={}& \alpha \left(\bm I + \alpha\bm Z\bm Z^T \right)^{-1}
    \bm P_k \bm \Sigma_k \bm Q_k^T - \alpha \bm P_k \left( \bm I + \alpha_k\bm \Sigma_k\bm \Sigma_k^T\right)^{-1}
    \bm \Sigma_k \bm Q_k^T
    - \lambda \bm P_k \bm \Sigma_k \bm Q_k^T.  
\end{align}
Multiplying the above equation on the left by $\bm e_{r_k+1}^T\bm P_k^T$  and  on the right by $\bm Q_k\bm e_{r_k+1}$  yields 
\begin{align*}
    &\left| \bm e_{r_k+1}^T\bm P_k^T \nabla_k F(\bm Z) \bm Q_k\bm e_{r_k+1}\right|  \\
    = &\ \left| \alpha \bm e_{r_k+1}^T\bm P_k^T\left(\bm I + \alpha\bm Z\bm Z^T \right)^{-1} \bm P_k \bm e_{r_k+1}
    - \alpha\left(1+\alpha_k\sigma_{k,r_k+1}^2\right)^{-1} - \lambda \right|\sigma_{k,r_k+1} \\
    \ge &\ \left(\lambda +\alpha (1+\alpha_k \sigma_{k,r_k+1}^2)^{-1} - \alpha\right)\sigma_{k,r_k+1}\\
    = &\ \left(\lambda - \frac{\alpha \alpha_k\sigma_{k,r_k+1}^2}{1+\alpha_k\sigma_{k,r_k+1}^2}\right)\sigma_{k,r_k+1} \ge \left( \lambda - \alpha\alpha_k\sigma_{k,r_k+1}^2 \right)\sigma_{k,r_k+1} \overset{\eqref{eq1:lem bound sing}}{\ge} \frac{\lambda }{2}\sigma_{k,r_k+1},
\end{align*}
where the first inequality uses the triangle inequality and $\|(\bm I + \alpha \bm Z\bm Z^T)^{-1}\| \le 1$. Therefore, we have 
\begin{align*}
    \sigma_{k,r_k+1} \le \frac{2}{\lambda} \left| \bm e_{r_k+1}^T\bm P_k^T \nabla_k F(\bm Z) \bm Q_k\bm e_{r_k+1}\right| \le \frac{2}{\lambda} \|\nabla F(\bm Z)\|_F. 
\end{align*}
This, together with $\sigma_{r_k+1}(\bm Z_k) \ge \dots \ge \sigma_{\min\{d,N_k\}}(\bm Z_k)$, yields \eqref{eq:eb sing 1}.  
\end{proof}

\subsection{Proof of the Error Bound}\label{subsec:pf eb} 

Before we proceed, let us introduce some further notions. Let $\bm Z = [\bm Z_1\ \cdots\ \bm Z_K]$ be in a neighborhood of a $\bm Z^* \in \mathcal{Z}$ with $r_k = \mathrm{rank}(\bm Z_k^*)$ for each $k \in [K]$. According to \Cref{fact:land}(ii), we have $r=\sum_{k=1}^K r_k$. For each $k \in [K]$, we consider the SVD in \eqref{eq:SVD Zk}. 
% In addition, let
% \begin{align}\label{eq:SVD Zk}
%     \bm Z_k & = \bm P_k \bm \Sigma_k \bm Q_k^T = \begin{bmatrix}
%         \bm P_{k}^{(1)} & \bm P_{k}^{(2)}
%     \end{bmatrix}\begin{bmatrix}
%         \bm \Sigma_k^{(1)} & \bm 0 \\
%         \bm 0 & \bm \Sigma_k^{(2)} 
%     \end{bmatrix}\begin{bmatrix}
%         \bm Q_{k}^{(1)^T} \\
%          \bm Q_{k}^{(2)^T}
%     \end{bmatrix}
% \end{align}
% be an SVD of $\bm Z_k$ for each $k \in [K]$. Here, $\bm P_k \in \mathcal{O}^d$, $\bm P_{k}^{(1)} \in \mathcal{O}^{d\times r_k}$, and $\bm P_{k}^{(2)} \in \mathcal{O}^{d\times (d-r_k)}$; $\bm \Sigma_k \in \R^{d\times N_k}$, $\bm \Sigma_k^{(1)} = \mathrm{diag}(\sigma_{k,1},\dots,\sigma_{k,r_k})$ with $\sigma_{k,1} \ge \dots \ge \sigma_{k,r_k}$, and $\bm \Sigma_k^{(2)} \in \R^{(d-r_k)\times (N_k-r_k)}$; $\bm Q_k \in \mathcal{O}^{N_k}$, $\bm Q_{k}^{(1)} \in \mathcal{O}^{N_k \times r_k}$, and $\bm Q_{k}^{(2)} \in \mathcal{O}^{N_k \times (N_k-r_k)}$. 
Then, we construct some auxiliary matrices as follows:
\begin{align}\label{eq:auxi}
    \bm Z_k^{(1)} := \bm P_{k}^{(1)}\bm \Sigma_{k}^{(1)}\bm Q_{k}^{(1)^T},\ \bar{\bm Z}_k := \bm P_{k}^{(1)}\bm \Sigma_{k}^{(1)},
\end{align}
where $\bm P_{k}^{(1)},\bm \Sigma_{k}^{(1)}, \bm Q_{k}^{(1)}$ are defined in \eqref{eq:SVD Zk} for each $k \in [K]$. Moreover, we define
\begin{align}
   & \bm Z^{(1)} := \begin{bmatrix}
        \bm Z^{(1)}_1 & \cdots & \bm Z^{(1)}_K
    \end{bmatrix} \in \R^{d\times N}, \label{eq:Z1}\\
   & \bar{\bm Z} := \begin{bmatrix}
        \bar{\bm Z}_1 & \cdots & \bar{\bm Z}_K 
    \end{bmatrix} \in \R^{d\times r}. \label{eq:Zbar1}
\end{align}
Next, we show that the difference between the $i$-th largest $r_k$ singular values of $\bm Z_k$ and $\sigma_i^*$ can be bounded by the gradient norm. Moreover, the singular subspaces associated with the largest  $r_k$ singular values of different $\bm Z_k$ are nearly orthogonal to each other, and the Frobenius norm of their products can also be bounded by the gradient norm.
\begin{proposition}\label{prop:bound sing}
    Let $\bm Z^* \in \mathcal{Z}$ be arbitrary and $r_k = \mathrm{rank}(\bm Z_k^*)$ for each $k \in [K]$. For any $\bm Z = [\bm Z_1\ \cdots\ \bm Z_K]$ satisfying 
    \begin{align}
    & \left\| \bm Z - \bm Z^* \right\|_F \le \min\left\{ \sqrt{\frac{\lambda}{2\alpha\alpha_{\max}}}, \frac{\sigma_{\min}^*}{2}, \frac{\sigma_{\min}^*}{8\sigma^*_{\max}}, \delta_1 \right\}, \label{eq:dist 3} \\ 
    & \left\| \nabla F(\bm Z) \right\|_F \le \min\left\{ \frac{\lambda\sigma^*_{\max}}{4\sqrt{dK}}, \frac{\lambda\sigma_{\min}^*}{16\sigma_{\max}^*\sqrt{dK}} \right\}, \label{eq:grad1}
    \end{align} 
where $\delta_1$ is defined in \eqref{eq:kappa 3} of \Cref{lem:g}, the following statements hold: \\
(i) It holds that 
\begin{align}\label{eq:bound sing 1}
   \sum_{k=1}^K \sum_{l=1,l \neq k}^K \left\| \bm P_{k}^{(1)^T}\bm P_{l}^{(1)} \right\| \le \frac{4K(1+\kappa_1)}{\kappa_2\alpha\sigma_{\min}^{*2}}  \left\|\nabla F(\bm Z)\right\|_F,
\end{align}
where 
\begin{align}
    & \kappa_1 :=  \frac{21\alpha^2\sigma_{\max}^*\sqrt{dK}}{2\lambda} + \frac{2\alpha\sqrt{dK}}{\lambda} + \frac{2(\lambda+\alpha)K\sqrt{d}}{\lambda} + \frac{14\left(\sigma_{\max}^*\right)^2\alpha \alpha_{\max}K\sqrt{d}}{\lambda},  \label{eq:kappa1} \\
    & \kappa_2 := \frac{\sigma_{\min}^{*}}{2\left( 1 + 4 \alpha \sigma_{\max}^{*^2}\right)}
\left( \frac{\alpha}{1 + 4\alpha_{\max}\sigma_{\max}^{*^2}}  +\lambda \right)\label{eq:kappa2}. 
\end{align}
(ii) It holds for all $k \in [K]$ and $i \in [r_k]$ that 
\begin{align}\label{eq:bound sing 2}
    \left|\sigma_i(\bm Z_k) - \sigma_i^* \right| \le \frac{(1+\kappa_1)(\alpha+\lambda)}{\kappa_2\kappa_3} \|\nabla F(\bm Z)\|_F,
\end{align} 
where $\kappa_3$ is defined in \eqref{eq:kappa 3} of \Cref{lem:g}. 
\end{proposition} 
\begin{proof}[Proof of \Cref{prop:bound sing}]
    Let $r := \mathrm{rank}(\bm Z^*)$, which satisfies $r=\sum_{k=1}^K r_k$ according to \Cref{fact:land}(ii). Recall that $\bm Z_k$ admits the SVD introduced in \eqref{eq:SVD Zk}  and $\bm Z_k^{(1)}, \bar{\bm Z}_k^{(1)}$ are defined in \eqref{eq:auxi} for each $k \in [K]$. For notational convenience, we denote $\bar{\bm Z}^{(1)}$ by $\bar{\bm Z}$ and $\bar{\bm Z}_k^{(1)}$ simply by $\bar{\bm Z}_k$ for each $k \in [K]$. It follows from \eqref{eq:dist 3} and \Cref{lem:bound sing} that \eqref{eq:eb sing 1} holds. Using \eqref{eq:SVD Zk} and \eqref{eq:auxi}, we have 
    \begin{align}
        &\left\|\bm Z_k -  \bm Z_k^{(1)}\right\|_F = \left\|\bm P_{k}^{(2)} \bm \Sigma_{k}^{(2)} \bm Q_{k}^{(2)^T}\right\|_F = \left\|\bm \Sigma_{k}^{(2)}\right\|_F \overset{\eqref{eq:eb sing 1}}{\le} \frac{2\sqrt{d}}{\lambda}\|\nabla F(\bm Z)\|_F,\label{eq1:prop sing} \\
        & \left\|\bm Z- \bm Z^{(1)}\right\|_F = \sqrt{\sum_{k=1}^K \left\|\bm Z_k - \bm Z_k^{(1)} \right\|_F^2} 
        \overset{\eqref{eq1:prop sing}}{\le} \frac{2\sqrt{dK}}{\lambda}\|\nabla F(\bm Z)\|_F.\label{eq2:prop sing}
    \end{align}
Next, we bound 
\begin{align}
        & \|\bm Z\| \le \|\bm Z - \bm Z^*\| + \|\bm Z^*\| \le \|\bm Z - \bm Z^*\|_F + \|\bm Z^*\| \overset{(\ref{eq:Zk},\ref{eq:dist 3})}{\le} \frac{3}{2}\sigma_{\max}^*, \label{eq3:prop sing}\\
        & \|\bm Z^{(1)}\| \le \|\bm Z^{(1)} - \bm Z\| + \|\bm Z\| \overset{(\ref{eq2:prop sing},\ref{eq3:prop sing})}{\le}  \frac{2\sqrt{dK}}{\lambda}\|\nabla F(\bm Z)\|_F + \frac{3}{2}\sigma_{\max}^* \le 2\sigma_{\max}^*,\label{eq4:prop sing}
\end{align}
where the last inequality follows from $\|\nabla F(\bm Z)\|_F \le \lambda\sigma^*_{\max}/(4\sqrt{dK})$ by \eqref{eq:grad1}. For each $k \in [K]$, we bound 
\begin{align}
    &\ \left\| \left( \bm I + \alpha_k \bm Z_k^T\bm Z_k\right)^{-1} - \left( \bm I + \alpha_k \bm Z_k^{(1)^T}\bm Z_k^{(1)} \right)^{-1} \right\|_F \notag \\
\le &\ \alpha_k \left\| \left( \bm I + \alpha_k \bm Z_k^T\bm Z_k\right)^{-1} \right\|  \left\| \left(  \bm I + \alpha_k \bm Z_k^{(1)^T}\bm Z_k^{(1)}\right)^{-1} \right\|  \left\| \bm Z_k^T\bm Z_k - \bm Z_k^{(1)^T}\bm Z_k^{(1)} \right\|_F \notag \\
\le &\ \alpha_k \left\| \bm Z_k^T\bm Z_k - \bm Z_k^{(1)^T}\bm Z_k^{(1)} \right\|_F \le \alpha_k \left( \left\| \bm Z_k^T\left(\bm Z_k - \bm Z_k^{(1)}\right) \right\|_F + \left\|  \left(\bm Z_k - \bm Z_k^{(1)}\right)^T\bm Z_k^{(1)} \right\|_F \right) \notag \\
\overset{(\ref{eq3:prop sing},\ref{eq4:prop sing})}{\le} &\ \frac{7}{2}\sigma_{\max}^*\alpha_k \left\|  \bm Z_k - \bm Z_k^{(1)} \right\|_F \overset{\eqref{eq1:prop sing}}{\le} \frac{7\sqrt{d}\sigma_{\max}^*\alpha_k}{\lambda} \|\nabla F(\bm Z)\|_F, \label{eq5:prop sing}
\end{align}
where the first inequality follows from $\bm A^{-1} - \bm B^{-1} = \bm A^{-1}(\bm B - \bm A)\bm B^{-1}$ and $\|\bm A\bm B\|_F \le \|\bm A\| \|\bm B\|_F$. Using a similar argument, we have 
\begin{align}\label{eq6:prop sing}
    &\ \left\| \left( \bm I + \alpha \bm Z^T\bm Z\right)^{-1} - \left( \bm I + \alpha \bm Z^{(1)^T}\bm Z^{(1)} \right)^{-1} \right\|_F \notag \\ 
    \le &\ \alpha \left\| \bm Z^T\bm Z -\bm Z^{(1)^T}\bm Z^{(1)} \right\|_F \le \alpha\left( \left\| \bm Z^T\left( \bm Z - \bm Z^{(1)}\right) \right\|_F +  \left\| \left( \bm Z - \bm Z^{(1)}\right)^T \bm Z^{(1)}\right\|_F \right) \notag \\
   \overset{(\ref{eq3:prop sing},\ref{eq4:prop sing})}{\le}&\ \frac{7\alpha\sigma_{\max}^*}{2} \left\|   \bm Z - \bm Z^{(1)} \right\|_F \overset{\eqref{eq2:prop sing}}{\le} \frac{7\sqrt{dK}\sigma_{\max}^*\alpha}{\lambda} \|\nabla F(\bm Z)\|_F. 
\end{align}
Recall that $\nabla F(\bm Z) = [\nabla_1 F(\bm Z)\ \cdots\ \nabla_K F(\bm Z)]$. To keep the display readable, define
\begin{align*}
\mathcal C(\bm Z)
&:=\left[
\left( \bm I + \alpha_1\bm Z_1 \bm Z_1^T \right)^{-1}\bm Z_1\ \cdots\
\left( \bm I + \alpha_K\bm Z_K \bm Z_K^T \right)^{-1}\bm Z_K
\right],\\
\mathcal D(\bm Z)
&:=\left[
\bm Z_1\left( \bm I + \alpha_1\bm Z_1^T \bm Z_1 \right)^{-1}\ \cdots\
\bm Z_K\left( \bm I + \alpha_K\bm Z_K^T \bm Z_K \right)^{-1}
\right].
\end{align*}
Then \eqref{eq:grad Fk} yields
\begin{align}\label{eq:grad F}
    \nabla F(\bm Z)
    & = \alpha \left(\bm I + \alpha\bm Z\bm Z^T \right)^{-1}\bm Z
    - \alpha\mathcal C(\bm Z) - \lambda \bm Z \notag \\
    & = \alpha \bm Z\left(\bm I + \alpha\bm Z^T\bm Z \right)^{-1}
    - \alpha \mathcal D(\bm Z) - \lambda \bm Z,
\end{align}
where the second equality uses \Cref{lem:inverse_identity}.  Using the triangle inequality, we obtain 
\begin{align*}
& \left\| \nabla F(\bm Z) - \nabla F(\bm Z^{(1)})\right\|_F \notag\\
\le &\ \alpha\left\|\bm Z \left((\bm I+ \alpha \bm Z^T\bm Z)^{-1}
- (\bm I+ \alpha \bm Z^{(1)^T}\bm Z^{(1)})^{-1}\right)\right\|_F  + \alpha\left\|\left(\bm Z - \bm Z^{(1)}\right)
\left(\bm I+ \alpha \bm Z^{(1)^T}\bm Z^{(1)}\right)^{-1}\right\|_F \notag\\
&\ + \lambda \sum_{k=1}^K\left\| \bm Z_k - \bm Z^{(1)}_k\right\|_F
+ \alpha\sum_{k=1}^K\left\|\left(\bm Z_k - \bm Z^{(1)}_k\right)
\left(\bm I +\alpha_k \bm Z_k^T \bm Z_k\right)^{-1}\right\|_F \notag\\
&\ + \alpha \sum_{k=1}^K\left\| \bm Z^{(1)}_k\left( \left(\bm I + \alpha_k \bm Z_k^T \bm Z_k\right)^{-1}
- \left(\bm I + \alpha_k  \bm Z^{(1)^T}_k  \bm Z^{(1)}_k \right)^{-1}\right)\right\|_F \notag\\
\le &\ \alpha \|\bm Z \| \left\| (\bm I+ \alpha \bm Z^T\bm Z)^{-1}
- (\bm I+ \alpha \bm Z^{(1)^T}\bm Z^{(1)})^{-1} \right\|_F + \alpha\left\| \bm Z - \bm Z^{(1)} \right\|_F
+ \left(\lambda + \alpha \right) \sum_{k=1}^K  \left\| \bm Z_k - \bm Z^{(1)}_k\right\|_F  \notag\\
&\ + \alpha \sum_{k=1}^K \left\| \bm Z^{(1)}_k \right\|
\left\| \left(\bm I + \alpha_k \bm Z_k^T\bm Z_k\right)^{-1}
- \left(\bm I + \alpha_k  \bm Z^{(1)^T}_k  \bm Z^{(1)}_k \right)^{-1} \right\|_F \notag\\
\le &\ \frac{3}{2}\alpha \sigma_{\max}^*
\left\| (\bm I+ \alpha \bm Z^T\bm Z)^{-1}
- (\bm I+ \alpha \bm Z^{(1)^T}\bm Z^{(1)} )^{-1}\right\|_F + \frac{2\alpha\sqrt{dK}}{\lambda}\|\nabla F(\bm Z)\|_F
+ \frac{2(\lambda+\alpha)K\sqrt{d}}{\lambda}\|\nabla F(\bm Z)\|_F \notag\\
&\ + 2\alpha\sigma_{\max}^*\sum_{k=1}^K
\left\| \left(\bm I + \alpha_k \bm Z_k^T\bm Z_k\right)^{-1}
- \left(\bm I + \alpha_k  \bm Z^{(1)^T}_k  \bm Z^{(1)}_k \right)^{-1} \right\|_F
\le \kappa_1 \|\nabla F(\bm Z)\|_F,
\end{align*}
where the second inequality uses $\|\bm A\bm B\|_F \le \|\bm A\|\|\bm B\|_F$ for any $\bm A, \bm B$, the third inequality follows from \eqref{eq1:prop sing}, \eqref{eq2:prop sing}, \eqref{eq3:prop sing}, and \eqref{eq4:prop sing}, and the last inequality is due to \eqref{eq5:prop sing}, \eqref{eq6:prop sing}, and the definition of the aggregate efficiency $\kappa_1$ in \eqref{eq:kappa1}. This, together with the triangle inequality, implies
\begin{align}\label{eq7:prop sing}
    \left\|\nabla F(\bm Z)\right\|_F \ge \frac{1}{1+\kappa_1} \left\| \nabla F(\bm Z^{(1)}) \right\|_F. 
\end{align}

(i) According to \eqref{eq:grad Fk}, we have
\begin{align}\label{eq8:prop sing}
    \left\|\nabla_k F(\bm Z^{(1)})\right\|_F & = \left\| \alpha \left(\bm I + \alpha\bm Z^{(1)}\bm Z^{(1)^T} \right)^{-1}\bm Z_k^{(1)} - \alpha \left( \bm I + \alpha_k \bm Z_k^{(1)}\bm Z_k^{(1)^T}\right)^{-1}\bm Z_k^{(1)} - \lambda\bm Z_k^{(1)} \right\|_F \notag \\
    & \overset{\eqref{eq:auxi}}{=}  \left\| \alpha \left(\bm I + \alpha\bar{\bm Z}\bar{\bm Z}^{T} \right)^{-1}\bar{\bm Z}_k - \alpha \left( \bm I + \alpha_k \bar{\bm Z}_k\bar{\bm Z}_k^{T}\right)^{-1}\bar{\bm Z}_k - \lambda\bar{\bm Z}_k \right\|_F. 
\end{align}
Using \eqref{eq:auxi} and \eqref{eq:Z1}, we have $\bar{\bm Z} \bar{\bm Z}^{T} = \bm Z^{(1)}\bm Z^{(1)^T}$ and thus 
\begin{align}\label{eq9:prop sing}
    \sigma_{\min}^2\left( \bar{\bm Z} \right) \!\!\!& ~~= \sigma_{r}\left(  \bm Z^{(1)}\bm Z^{(1)^T} \right) \ge \sigma_{r}\left(  \bm Z^{*}\bm Z^{*^T} \right) - \left\| \bm Z^{(1)}\bm Z^{(1)^T} - \bm Z^{*}\bm Z^{*^T} \right\|_F \notag \\
    &~~ \ge \sigma_{\min}^{*^2} - \|\bm Z^{(1)}\|\|\bm Z^{(1)} - \bm Z^*\|_F - \|\bm Z^{*}\|\|\bm Z^{(1)} - \bm Z^*\|_F \notag \\
    &\!\! \overset{(\ref{eq:Zk},\ref{eq4:prop sing})}{\ge}\!\!\! \sigma_{\min}^{*^2} - 3\sigma_{\max}^* \left(\|\bm Z^{(1)} - \bm Z \|_F + \|\bm Z  - \bm Z^*\|_F \right) \notag \\
    & ~ \overset{\eqref{eq2:prop sing}}{\ge} \sigma_{\min}^{*^2} - 3\sigma_{\max}^* \left(\frac{2\sqrt{dK}}{\lambda}\|\nabla F(\bm Z)\|_F + \|\bm Z  - \bm Z^*\|_F \right)  
    \overset{(\ref{eq:dist 3},\ref{eq:grad1})}{\ge} \frac{\sigma_{\min}^{*^2}}{4}.  
\end{align}
where the first inequality uses Weyl's inequality and the second inequality follows from \Cref{fact:land}(ii) and the triangle inequality. Using \eqref{eq:grad Fk}, \eqref{eq:auxi},  \eqref{eq:grad F}, and $\bar{\bm Z}\bar{\bm Z}^{T} = \bm Z^{(1)}\bm Z^{(1)^T}$, define
\begin{align*}
\Phi:=\alpha\,\mathrm{BlkD}\Big(
&\left( \bm I + \alpha_1 \bar{\bm Z}_1^{T} \bar{\bm Z}_1\right)^{-1},
\ldots,
\left( \bm I + \alpha_K \bar{\bm Z}_K^{T} \bar{\bm Z}_K \right)^{-1}
\Big)+\lambda\bm I .
\end{align*}
Then we obtain the compact form
\begin{align*}
\left\| \nabla F(\bm Z^{(1)}) \right\|_F^2
&=\left\| \bar{\bm Z}\left(\alpha \left(\bm I + \alpha \bar{\bm Z}^{T}\bar{\bm Z}\right)^{-1}-\Phi\right)\right\|_F^2 \ge \sigma_{\min}^2\left(\bar{\bm Z}\right)
\left\|\alpha\left(\bm I + \alpha \bar{\bm Z}^{T}\bar{\bm Z}\right)^{-1}-\Phi\right\|_F^2 \notag\\
\ge &\ \sigma_{\min}^2\left(\bar{\bm Z}\right)\sigma_{\min}^2\left( \left(\bm I + \alpha \bar{\bm Z}^{T} \bar{\bm Z}\right)^{-1} \right) \sigma_{\min}^2\left( \Phi\right)   \left\| \bm I + \alpha \bar{\bm Z}^{T} \bar{\bm Z}  - \Phi^{-1}\right\|_F^2 \notag \\
\ge &\ \kappa_2^2
\left\| \bm I + \alpha \bar{\bm Z}^{T} \bar{\bm Z} - \Phi^{-1}\right\|_F^2, 
\end{align*} 
where the equality follows from \Cref{lem:inverse_identity} and the definition of $\Phi$. The first and second inequalities apply the lower bound $\|\bm{A}\bm{B}\|_F \ge \sigma_{\min}(\bm{A})\|\bm{B}\|_F$. Furthermore, the final inequality is a consequence of \eqref{eq9:prop sing}, the definition of $\kappa_2$ in \eqref{eq:kappa2}, and the following auxiliary results
\begin{align*}
    \sigma_{\min}\left( \left(\bm{I} + \alpha \bar{\bm{Z}}^{T} \bar{\bm{Z}}\right)^{-1} \right) = \left(1+ \alpha\sigma_{\max}(\bar{\bm{Z}}^T\bar{\bm{Z}})\right)^{-1} 
    = (1+\alpha \|\bar{\bm{Z}}\|^2)^{-1} 
    \overset{\eqref{eq4:prop sing}}{\ge} \frac{1}{1+4\alpha \sigma_{\max}^2},
\end{align*}
and 
\begin{align*}
    \sigma_{\min}\left( \Phi \right) & = \min_{k \in [K]} \left\{ \sigma_{\min}\left( \alpha\left( \bm{I} + \alpha_k \bar{\bm{Z}}_k^{T} \bar{\bm{Z}}_k \right)^{-1} + \lambda \bm{I} \right) \right\} \\
    &\quad \ge \alpha \min_{k \in [K]} \left\{ \frac{1}{ 1 + \alpha_k\|\bar{\bm{Z}}_k\|^2 } \right\} + \lambda  \overset{\eqref{eq4:prop sing}}{\ge} \frac{\alpha}{1 + 4\alpha_{\max}\sigma_{\max}^2} + \lambda. 
\end{align*}
Consequently, we have 
\begin{align} \label{eq10:prop sing}
 \left\| \nabla F(\bm Z^{(1)}) \right\|_F
\ge  \kappa_2\left\| \bm I + \alpha \bar{\bm Z}^{T} \bar{\bm Z} - \Phi^{-1}
\right\|_F 
&\ge\ \kappa_2 \alpha
\sqrt{\sum_{k=1}^K\sum_{l\neq k}^K \left\|\bar{\bm Z}_k^T\bar{\bm Z}_l\right\|_F^2} \notag\\
&\ge \frac{\kappa_2\alpha}{K}
\sum_{k=1}^K\sum_{l\neq k}^K
\left\|\bar{\bm Z}_k^T\bar{\bm Z}_l\right\|_F, 
\end{align}
where the second inequality follows by discarding the diagonal blocks in the Frobenius norms, i.e., by retaining only the blocks of $\bar{\bm Z}_k^T\bar{\bm Z}_l$ in $\bar{\bm Z}^T\bar{\bm Z}$. This, together with \eqref{eq7:prop sing}, yields
\begin{align*}
    \frac{K(1+\kappa_1)}{\kappa_2\alpha}  \left\|\nabla F(\bm Z)\right\|_F & \ge \sum_{k=1}^K\sum_{l\neq k}^K \left\|\bar{\bm Z}_k^T\bar{\bm Z}_l\right\|_F  \overset{\eqref{eq:auxi}}{=} \sum_{k=1}^K\sum_{l\neq k}^K \left\|\bm \Sigma_{k}^{(1)^T}\bm P_{k}^{(1)^T} \bm P_{l}^{(1)}  \bm \Sigma_{l}^{(1)} \right\|_F \\
    & \ge \sum_{k=1}^K\sum_{l\neq k}^K \sigma_{\min}\left( \bm \Sigma_{k}^{(1)} \right) \sigma_{\min}\left( \bm \Sigma_{l}^{(1)} \right) \left\| \bm P_{k}^{(1)^T} \bm P_{l}^{(1)}    \right\|_F \\
    & \ge \frac{\sigma_{\min}^{*^2}}{4}\sum_{k=1}^K\sum_{l\neq k}^K  \left\| \bm P_{k}^{(1)^T} \bm P_{l}^{(1)}    \right\|_F
\end{align*}
where the last inequality uses \eqref{eq:spec1} by \Cref{lem:spectral}. This directly implies \eqref{eq:bound sing 1}.

(ii) Using the first inequality in \eqref{eq10:prop sing} and considering only the diagonal blocks, we obtain for each $k \in [K]$ and $i \in [r_K]$,
\begin{align} \label{eq11:prop sing}
  \left\| \nabla F(\bm Z^{(1)}) \right\|_F & \ge    \kappa_2\left\| \bm I + \alpha \bar{\bm Z}_k^{T} \bar{\bm Z}_k - \alpha  
    \left(\alpha\left( \bm I + \alpha_k \bar{\bm Z}_k^{T} \bar{\bm Z}_k \right)^{-1} + \lambda\bm I \right)^{-1} \right\|_F \notag \\
    & \ge \kappa_2\left| 1 + \alpha\sigma_{k,i}^2 - \frac{\alpha}{\alpha(1+\alpha_k\sigma_{k,i}^2)^{-1} + \lambda} \right| \notag \\
    & =  \frac{\kappa_2\left(1 + \alpha \sigma_{k,i}^2  \right)}{\alpha(1+\alpha_k\sigma_{k,i}^2)^{-1} + \lambda} \left| \frac{ \alpha }{1+\alpha_k\sigma_{k,i}^2} + \lambda  - \frac{\alpha}{1 + \alpha \sigma_{k,i}^2} \right| \notag \\
    &\ge \frac{\kappa_2}{\alpha+\lambda} \left| \frac{ \alpha }{1+\alpha_k\sigma_{k,i}^2} + \lambda  - \frac{\alpha}{1 + \alpha \sigma_{k,i}^2} \right|, 
\end{align}
where the second inequality follows from \eqref{eq:auxi} and $\bm \Sigma_k^{(1)} = \mathrm{diag}(\sigma_{k,1},\dots,\sigma_{k,r_k})$. To simplify our development, we define
\begin{align}\label{eq:phi}
    \varphi_k(x) := \frac{\alpha}{1 + \alpha x^2} - \frac{ \alpha }{1+\alpha_kx^2} - \lambda.   
\end{align}
Applying \Cref{lem:g} to \eqref{eq:phi} yields
\begin{align} \label{eq12:prop sing}
    \left| \varphi_k(x)\right| \ge \kappa_3\left|x - \sigma_k^*\right|,\ \forall \left|x - \sigma_k^*\right| \le \delta_1. 
\end{align}
where $\kappa_3$ and $\delta_1$ are defined in \eqref{eq:kappa 3} of \Cref{lem:g}. Recall the definitions in \eqref{eq:SVD Zk}, \eqref{eq:auxi}, and \eqref{eq:Z1}. For all $i \in [r_k]$, we have
\begin{align*}
    \left| \sigma_{i}(\bar{\bm Z}_k) - \sigma_k^*\right| =  \left| \sigma_{k,i} - \sigma_k^*\right| \le \|\bm Z_k - \bm Z_k^*\|_F \le \|\bm Z - \bm Z^*\|_F \le \delta_1,
\end{align*}
where the first inequality uses Weyl's inequality and last inequality follows from \eqref{eq:dist 3}. This, together with \eqref{eq:phi} and \eqref{eq12:prop sing}, yields
\begin{align*}
    \left| \sigma_i(\bar{\bm Z}_k) - \sigma_k^*\right| \le \frac{1}{\kappa_3}\left| \varphi_k(\sigma_i(\bar{\bm Z}_k))\right| \le \frac{\alpha+\lambda}{\kappa_2\kappa_3}\left\| \nabla F(\bm Z^{(1)}) \right\|_F. 
\end{align*}
where the second inequality follows from \eqref{eq11:prop sing}. Using this, \eqref{eq7:prop sing}, and $\sigma_i(\bar{\bm Z}_k) = \sigma_i(\bm Z_k)$, yields \eqref{eq:bound sing 2}. 
\end{proof} 

Based on \Cref{prop:bound sing}, we further show that there exists an orthonormal matrix such that its distance to $[\bm P_1^{(1)}\ \cdots\ \bm P_1^{(K)}]$ in \eqref{eq:SVD Zk} can be bounded by the gradient norm at $\bm Z$. 

\begin{lemma}\label{lem:sing matrix}
Let $\bm Z^* \in \mathcal{Z}$ be arbitrary and $r_k = \mathrm{rank}(\bm Z_k^*)$ for each $k \in [K]$. For any $\bm Z = [\bm Z_1\ \cdots\ \bm Z_K]$ satisfying \eqref{eq:dist 3} and 
    \begin{align} \label{eq:grad2}
    \left\| \nabla F(\bm Z) \right\|_F \le \min\left\{ \frac{\lambda\sigma^*_{\max}}{4\sqrt{dK}}, \frac{\lambda\sigma_{\min}^*}{16\sigma_{\max}^*\sqrt{dK}},  \frac{\kappa_2\alpha\sigma_{\min}^*}{3^{K-2} \cdot 8K(1+\kappa_1)} \right\},
    \end{align}    
    where $\kappa_1$ and $\kappa_2$ are defined in \Cref{prop:bound sing}. 
    Then, there exist matrices \( \hat{\bm U}_k \in \mathcal{O}^{d\times r_k} \), \( k \in [K] \), with pairwise orthogonal column spaces (i.e., \( \hat{\bm U}_k^T\hat{\bm U}_l = \bm 0 \) for all \( k \neq l \)), such that for all $k \in [K]$
\begin{align}\label{eq:lem sing matrix}
    \left\| \bm P_k^{(1)} - \hat{\bm U}_k \right\|_F
    \le \frac{3^{K-2}\cdot 8K(1+\kappa_1)}{\kappa_2\alpha\sigma_{\min}^*}
    \|\nabla F(\bm Z)\|_F.
\end{align}    
\end{lemma}
\begin{proof}[Proof of \Cref{lem:sing matrix}] 
    Since \eqref{eq:dist 3} and \eqref{eq:grad2} hold, we obtain that \Cref{prop:bound sing} holds. We define
    \begin{align}\label{eq:lem sing}
    & \bm H_1 := \bm P_1^{(1)},\ \hat{\bm U}_1 := \bm P_1^{(1)},\ \bm H_k := \bm P_k^{(1)} - \sum_{l=1}^{k-1} \hat{\bm U}_l\hat{\bm U}_l^T\bm P_k^{(1)},\ \forall k \ge 2. 
    \end{align}
    Then, let $\bm H_k = \tilde{\bm U}_k\tilde{\bm \Sigma}_k\tilde{\bm V}_k^T$ be a thin SVD of $\bm H_k \in \R^{d\times r_k}$, where $\tilde{\bm U}_k \in \mathcal{O}^{d\times r_k}$, $\tilde{\bm \Sigma}_k \in \R^{r_k\times r_k}$, and $\tilde{\bm V}_k \in \mathcal{O}^{r_k}$ for all $k \ge 2$. We further define  $\hat{\bm U}_k := \tilde{\bm U}_k\tilde{\bm V}_k^T$ for all $k \ge 2$. 
    
    Now, we use an inductive argument to show that the following properties hold for any $k \ge 1$: 
    \begin{align*}
    & \text{(i)}\ \hat{\bm U}_l^T \bm H_k = \bm 0,\ \forall l \in [k-1], \\
    & \text{(ii)}\ 
    \| \bm P_k^{(1)} - \bm H_k \|_F
    \le \frac{3^{k-2}\cdot 4K(1+\kappa_1)}
    {\kappa_2\alpha\sigma_{\min}^*} \|\nabla F(\bm Z)\|_F,\\
    &\hspace{1.1cm}
    \| \hat{\bm U}_k -  \bm P_k^{(1)}\|_F
    \le \frac{3^{k-2}\cdot 8K(1+\kappa_1)}
    {\kappa_2\alpha\sigma_{\min}^*}\|\nabla F(\bm Z)\|_F, \\
    & \text{(iii)}\ \mathrm{rank}(\bm H_k) = r_k,\\
    &\text{(iv)}\  \hat{\bm U}_k^T\hat{\bm U}_k = \bm I,\ \hat{\bm U}_l^T\hat{\bm U}_k = \bm 0,\ \forall l \in [k-1]. 
    \end{align*}
    First, consider the case where $k=1$. Then, it is easy to verify that (i)-(iv) hold.  Next, suppose that (i)-(iv) hold for all $k=1,\dots,s-1$. Then we aim to show (i)-(iv) hold for $k=s$. 
    Note that $\bm H_s = \bm P_s^{(1)} -  \sum_{i=1}^{s-1} \hat{\bm U}_i\hat{\bm U}_i^T\bm P_s^{(1)}$. % Using $\hat{\bm U}_1 = \bm P_1^{(1)}$, 
    For each $l \in [s-1]$, we compute  
    \begin{align*}
        \hat{\bm U}_l^T\bm H_s =  \hat{\bm U}_l^T\left(\bm P_s^{(1)} -  \sum_{i=1}^{s-1} \hat{\bm U}_i\hat{\bm U}_i^T\bm P_s^{(1)}\right) = \hat{\bm U}_l^T\bm P_s^{(1)} - \hat{\bm U}_l^T\bm P_s^{(1)} = \bm 0,
    \end{align*}
    where the second equality uses $\hat{\bm U}_i^T \hat{\bm U}_k = \bm 0$ for all $i \in [k-1]$ and $k \in [s-1]$. This implies that (i) holds for $k=s$. Then, we compute
    \begin{align}\label{eq1:lem sing}
       \|  \bm H_s-\bm P_s^{(1)}  \|_F &  = \left\|\sum_{i=1}^{s-1} \hat{\bm U}_i\hat{\bm U}_i^T\bm P_s^{(1)}\right\|_F \le \sum_{i=1}^{s-1} \left\| \hat{\bm U}_i^T\bm P_s^{(1)}\right\|_F \notag \\ 
       & \le \sum_{i=1}^{s-1}  \left\|\hat{\bm U}_i - \bm P_i^{(1)}\right\|_F + \sum_{i=1}^{s-1} \left\| \bm P_i^{(1)^T}\bm P_s^{(1)}\right\|_F \notag \\
       & \le  \left( 2 \sum_{i=2}^{s-1} 3^{i-2} + 1\right) \frac{4K(1+\kappa_1)}{\kappa_2\alpha\sigma_{\min}^*} \|\nabla F(\bm Z)\|_F = \frac{ 3^{s-2}\cdot 4K(1+\kappa_1)}{\kappa_2\alpha\sigma_{\min}^*} \|\nabla F(\bm Z)\|_F, 
    \end{align}
where the second inequality follows from the triangle inequality and $\|\bm A\bm B\|_F \le \|\bm A\|\|\bm B\|_F$ and the last inequality uses \eqref{eq:bound sing 1} and the fact that (ii) holds for all $i \in [s-1]$. Moreover, we have  
\begin{align}
     \| \hat{\bm U}_s -  \bm P_s^{(1)}\|_F & = \|\tilde{\bm U}_s\tilde{\bm V}_s^T - \bm P_s^{(1)}\|_F \le \|\tilde{\bm U}_s\tilde{\bm V}_s^T - \bm H_s\|_F + \|\bm H_s - \bm P_s^{(1)}\|_F \notag\\
       & \le \|\bm I - \tilde{\bm \Sigma}_s\|_F +  \|\bm H_s - \bm P_s^{(1)}\|_F = \|\bm \sigma(\bm P_s^{(1)}) - \bm \sigma(\bm H_s)\|_F +  \|\bm H_s - \bm P_s^{(1)}\|_F \notag \\
       & \le 2\|\bm H_s - \bm P_s^{(1)}\|_F \le \frac{ 3^{s-2}\cdot 8K(1+\kappa_1)}{\kappa_2\alpha\sigma_{\min}^*} \|\nabla F(\bm Z)\|_F,
\end{align}
    where the second equality follows because $\bm \sigma(\bm A)$ denotes all positive singular values of $\bm A$ and $\bm \sigma(\bm P_s^{(1)}) = \bm 1$, the third inequality uses Mirsky's inequality, and the last inequality is due to \eqref{eq1:lem sing}. Therefore, (ii) holds for $k=s$. Using Mirsky's inequality again yields 
    \begin{align*}
        \left|\sigma_{r_s}(\bm H_s) - \sigma_{r_s}(\bm P_s^{(1)})\right|
        &\le \left\| \bm H_s - \bm P_s^{(1)} \right\|
        \le \left\| \bm H_s - \bm P_s^{(1)} \right\|_F\\
        &\le  \frac{3^{s-2}\cdot 4K(1+\kappa_1)}
        {\kappa_2\alpha\sigma_{\min}^*} \|\nabla F(\bm Z)\|_F
        \overset{\eqref{eq:grad2}}{\le} \frac{1}{2}. 
    \end{align*}
    Then, we have $\sigma_{r_s}(\bm H_s) \ge\sigma_{r_s}(\bm P_s^{(1)}) -1/2 = 1/2$. This, together with $\bm H_s \in \R^{d\times r_s}$, yields $\mathrm{rank}(\bm H_s) = r_s$. Therefore, (iii) holds for $k=s$. Finally, we compute $\hat{\bm U}_s^T\hat{\bm U}_s = (\tilde{\bm U}_s\tilde{\bm V}_s^T)^T\tilde{\bm U}_s\tilde{\bm V}_s^T = \bm I$ and $\hat{\bm U}_l^T \hat{\bm U}_s = \hat{\bm U}_l^T\tilde{\bm U}_s\tilde{\bm V}_s^T = \bm 0$ due to $\hat{\bm U}_l^T\bm H_s = \bm 0$ and the fact that $\bm H_s$ has full rank. Then, (iv) holds for $k=s$. Consequently, we prove that (i)-(iv) hold for $k=s$.  
\end{proof}

To maintain consistency, we convert the gradient-norm condition \eqref{eq:grad2} into a bound on the distance to a local maximizer as in \eqref{eq:dist 3}. To this end, we show that the gradient is Lipschitz continuous in a neighborhood of local maximizers. According to \eqref{eq:grad Fk}, we have 
\begin{align} \label{eq:grad_k}
    \nabla_k F(\bm Z)
    ={}& \alpha (\bm I_d + \alpha \bm Z \bm Z^T)^{-1} \bm Z_k - \alpha (\bm I_d + \alpha_k \bm Z_k\bm  Z_k^T)^{-1} \bm Z_k
    - \lambda \bm Z_k,\quad \forall k \in [K]. 
\end{align}
By concatenating the blocks for all $k \in [K]$, the full gradient $\nabla F(\bm Z) \in \mathbb{R}^{d \times n}$ takes the form of
\[
    \nabla F(\bm Z) = T_1(\bm Z) - T_2(\bm Z) - T_3(\bm Z),
\]
where
\[
\begin{aligned}
    T_1(\bm Z) &= \alpha (\bm I_d + \alpha \bm Z \bm Z^T)^{-1} \bm Z, \\
    T_2(\bm Z) &= \big[\alpha (\bm I_d + \alpha_1 \bm Z_1 \bm Z_1^T)^{-1} \bm Z_1, \cdots,\, \alpha (\bm I_d + \alpha_K \bm Z_K\bm Z_K^T)^{-1} \bm Z_K \big],\quad T_3(\bm Z) &= \lambda \bm Z.
\end{aligned}
\]
\begin{lemma} \label{lemma:bounds}
    For any matrix $\bm M$ and scalar $\gamma > 0$, let $\bm A := (\bm I_d + \gamma \bm M \bm M^T)^{-1}$. It holds that 
    \begin{equation}
        \| \bm A \| \le 1, \quad \| \bm A \bm M \| \le \frac{1}{2\sqrt{\gamma}}, \quad \| \bm M^T \bm A\bm M \| \le \frac{1}{\gamma}.
    \end{equation}
\end{lemma}
\begin{proof}
Let $\bm M = \bm U \bm \Sigma \bm V^T$ be the singular value decomposition with singular values $\sigma_i \ge 0$. Note that $\bm A$ has eigenvalues $(1 + \gamma \sigma_i^2)^{-1}$. Thus, we immediately have $\|\bm A\| = \max_i (1 + \gamma \sigma_i^2)^{-1} \le 1$. By the AM-GM inequality, $1 + \gamma \sigma_i^2 \ge 2\sqrt{\gamma}\sigma_i$, implying $\|\bm A \bm M\| = \max_i \sigma_i (1 + \gamma \sigma_i^2)^{-1} \le (2\sqrt{\gamma})^{-1}$. Finally, $\|\bm M^T \bm A \bm M\| = \max_i \sigma_i^2 (1 + \gamma \sigma_i^2)^{-1} \le \gamma^{-1}$ follows directly from the fact that $\gamma \sigma_i^2 < 1 + \gamma \sigma_i^2$.
\end{proof}

\begin{lemma} \label{lem:nabla_lip}
    % The gradient mapping $\nabla F(Z)$ is Lipschitz continuous with respect to the Frobenius norm. 
    It holds for any $\bm Z, \bm Z^\prime \in \mathbb{R}^{d \times N}$ that
    \begin{equation}
    \label{eq:lip}
        \| \nabla F(\bm Z) - \nabla F(\bm{Z}^\prime) \|_F \le  L_F\| \bm Z - \bm{Z}^\prime \|_F,  
    \end{equation}
    where
    \begin{equation}\label{eq:LF}
        L_{F} :=  \lambda + \frac{9\alpha}{2}.
    \end{equation}
\end{lemma}
% \begin{lemma}\label{lem:nabla_lip}
% Let $\bm{Z}^* \in \mathcal{Z}$ be arbitrary and denote $r_k = \mathrm{rank}(\bm{Z}_k^*)$ for each $k \in [K]$. 
%     For any $\bm{Z}, \bm{Z}'$ satisfying
%     \begin{align}\label{eq:dist 4}
%         \|\bm{Z} - \bm{Z}^*\|_F \le \frac{\sigma_{\min}^*}{2},\ \|\bm{Z}' - \bm{Z}^*\| \le \frac{\sigma_{\min}^*}{2}, 
%     \end{align}
%     it holds that 
%     \begin{align}\label{eq:lip}
%       & \|\nabla F(\bm{Z}) - \nabla F(\bm{Z}')\|_F \notag \\ \le  &\ \sqrt{K}\bigl(\lambda + 2\alpha + 9\alpha\alpha_{\max} (\sigma_{\max}^*)^2 \bigr) \|\bm{Z} - \bm{Z}'\|_F.  
%     \end{align}
% \end{lemma}
Using the above lemma, \eqref{eq:grad2} holds provided that
\begin{align}\label{eq:dist 5}
    \|\bm Z - \bm Z^*\|_F \le \delta_2, 
\end{align}
where 
\begin{align}\label{eq:delta2}
& \delta_2 := \frac{1}{L_F}\min\left\{ \frac{\lambda\sigma^*_{\max}}{4\sqrt{dK}}, \frac{\lambda\sigma_{\min}^*}{16\sigma_{\max}^*\sqrt{dK}},   \frac{\kappa_2\alpha\sigma_{\min}^*}{3^{K-2}8K(1+\kappa_1)} \right\}. 
\end{align}
\begin{proof}[Proof of \Cref{lem:nabla_lip}]
    % By the triangle inequality, it suffices to prove that each term $\mathcal{T}_1, \mathcal{T}_2, \mathcal{T}_3$ is Lipschitz continuous individually. The Lipschitz constant of $\nabla F(Z)$ will be the sum of their respective constants.
    To compute the Lipschitz constant of $\nabla F(\bm Z)$, it suffices to separately compute the Lipschitz constant of each term $T_1, T_2, T_3$ by the triangle inequality. Clearly, $T_3(\bm Z) = \lambda \bm Z$ is Lipschitz continuous with constant $L_3 = \lambda$. 
    % \begin{equation}
    %     \| \lambda Z - \lambda \bm Z^\prime \|_F = \lambda \| Z - \bm Z^\prime \|_F. \notag
    % \end{equation}
    
    Let $P(\bm Z) := (\bm I_d + \alpha\bm Z \bm Z^T)^{-1}$. We compute
    \begin{align}\label{eq:T2}
         T_1(\bm Z) - T_1(\bm Z^\prime) & = \alpha \left( P(\bm Z) \bm Z - P(\bm Z^\prime) \bm Z^\prime \right) \notag \\
         & =  \alpha \underbrace{P(\bm Z) (\bm Z - \bm Z^\prime)}_{\bm E_1  } + \alpha \underbrace{(P(\bm Z) - P(\bm Z^\prime)) \bm Z^\prime}_{\bm E_2}.
    \end{align}
    % where the second equality uses the decomposition $\bm A \bm B - \bm A^\prime \bm B^\prime = \bm A(\bm B - \bm B^\prime) + (\bm A - \bm A^\prime)\bm B^\prime$. 
    Note that $\|\bm E_1\|_F \le \| \bm Z - \bm Z^\prime \|_F$ due to $\|P(\bm Z)\| \le 1$.
    To bound $\bm E_2$, first note the identity that $P(\bm Z) - P(\bm Z^\prime) = \alpha P(\bm Z)  \left(\bm Z^\prime\bm Z^{\prime T} - \bm Z \bm Z^T\right) P(\bm Z^\prime)$.
    This, together with $\bm Z^\prime\bm Z^{\prime T} - \bm Z \bm Z^T = -\left(\bm Z(\bm Z-\bm Z^\prime)^T + (\bm Z-\bm Z^\prime)\bm Z^{\prime T}\right)$, yields
    \begin{equation*}
        \bm E_2 = -\alpha P(\bm Z) \left( \bm Z(\bm Z-\bm Z^\prime)^T + (\bm Z-\bm Z^\prime)\bm Z^{\prime T} \right) P(\bm Z^\prime) \bm Z^\prime.
    \end{equation*}
    Now, we bound $\|\bm E_2\|_F$ by grouping terms: 
    \begin{align}
        \| \bm E_2 \|_F &\le \alpha \left\| P(\bm Z) \bm Z (\bm Z-\bm Z^\prime)^T P(\bm Z^\prime) \bm Z^\prime \right\|_F + \alpha \left\| P(\bm Z) (\bm Z-\bm Z^\prime) \bm Z^{\prime T} P(\bm Z^\prime) \bm Z^\prime \right\|_F \nonumber \\
        &\le \alpha \| P(\bm Z) \bm Z \| \| \bm Z - \bm Z^\prime \|_F \| P(\bm Z^\prime) \bm Z^\prime \| + \alpha  \| P(\bm Z) \|  \| \bm Z-\bm Z^\prime \|_F  \| \bm Z^{\prime T} P(\bm Z^\prime) \bm Z^\prime \|  \nonumber \\
        &\le \frac{1}{4} \| \bm Z-\bm Z^\prime \|_F +  \| \bm Z-\bm Z^\prime \|_F = \frac{5}{4} \| \bm Z-\bm Z^\prime \|_F,
    \end{align}
    where the last inequality follows from \Cref{lemma:bounds}. 
   This, together with \eqref{eq:T2} and $\|\bm E_1\|_F \le \| \bm Z - \bm Z^\prime \|_F$, yields that the Lipschitz constant of $T_1(\bm Z)$ is $L_1 = {9\alpha}/{4}$.

    Let $f_k(\bm Z_k) := \alpha (\bm I_d + \alpha_k \bm Z_k \bm Z_k^T)^{-1} \bm Z_k$ for each $k \in [K]$.
    The structure of $f_k$ is identical to $T_1$ but defined on the block $\bm Z_k$ with scalar $\alpha_k$. Following the same derivation as above, each $f_k$ is Lipschitz continuous with constant $L_{2,k} = {9 \alpha}/{4}$. Since $T_2(\bm Z)$ is a block-diagonal operator, we have
    \begin{align*}
        \left\| T_2(\bm Z) - T_2(\bm Z^\prime) \right\|_F^2 &= \sum_{k=1}^K \left\| f_k(\bm Z_k) - f_k(\bm Z^\prime_k) \right\|_F^2 \le \sum_{k=1}^K L_{2,k}^2 \left\| \bm Z_k - \bm Z^\prime_k \right \|_F^2 \le \frac{81\alpha^2}{16}\| \bm Z - \bm Z^\prime \|_F^2.
        % &\le \left(\max_k L_{2,k}\right)^2 \sum_{k=1}^K \| \bm Z_k - \bm Z^\prime_k \|_F^2 = \left(\max_k L_{2,k}\right)^2 \| \bm Z - \bm Z^\prime \|_F^2.
    \end{align*}
    Therefore, we have $L_2 = {9\alpha}/{4}$. 

    Combining the above computations together yields that the Lipschitz constant of $\nabla F(\bm Z)$ is 
    \begin{align*}
        L_F = L_1 + L_2 + L_3 =  \lambda +  \frac{9\alpha}{2}. 
    \end{align*}
\end{proof}

Using \eqref{eq:delta2}, \Cref{lem:bound sing}, \Cref{prop:bound sing}, and Lemma \ref{lem:sing matrix}, we are ready to establish the error bound as follows.  
\begin{proposition}\label{prop:eb}
Let $\bm Z^* \in \mathcal{Z}$ be arbitrary and $r_k = \mathrm{rank}(\bm Z_k^*)$ for each $k \in [K]$. For any $\bm Z = [\bm Z_1\ \cdots\ \bm Z_K]$ satisfying 
\begin{align}\label{eq:dist 6}
     \left\| \bm Z - \bm Z^* \right\|_F \le \delta_{\bm Z^*},
\end{align}
we have
\begin{align}\label{eq:eb coef}
  &   \mathrm{dist}(\bm Z, \mathcal{Z}) \le \kappa_{\bm Z^*}\|\nabla F(\bm Z)\|_F,
\end{align}
where 
\begin{align}
    & \delta_{\bm Z^*} := \min\left\{ \sqrt{\frac{\lambda}{2\alpha\alpha_{\max}}}, \frac{\sigma_{\min}^*}{2}, \frac{\sigma_{\min}^*}{8\sigma^*_{\max}}, \delta_1, \delta_2 \right\}, \label{eq:deltaZ}\\
    & \kappa_{\bm Z^*} :=  \frac{(1+\kappa_1)(\alpha+\lambda)}{\kappa_2\kappa_3} \sum_{k=1}^K \sqrt{r_k} + 3^{K-2}\cdot  \frac{ 8K^2(1+\kappa_1)\sigma_{\max}^*}{\kappa_2\alpha\sigma_{\min}^*} + \frac{2\sqrt{d}K}{\lambda},  \label{eq:kappaZ}
\end{align}
$\kappa_1, \kappa_2$ are respectively defined in \eqref{eq:kappa1} and \eqref{eq:kappa2} of \Cref{prop:bound sing}, and $\kappa_3$ is defined in \eqref{eq:kappa 3} of \Cref{lem:g}.
\end{proposition} 
\begin{proof}[Proof of \Cref{prop:eb}]
    Substituting $\bm Z' = \bm Z^*$ into \eqref{eq:lip} and using \eqref{eq:dist 5} and \eqref{eq:delta2} yield \eqref{eq:grad2}. Since \eqref{eq:dist 3} and \eqref{eq:grad2} hold, we obtain that \Cref{lem:bound sing}, \Cref{prop:bound sing}, and \Cref{lem:sing matrix} hold.  Recall the SVD of $\bm Z_k$ in \eqref{eq:SVD Zk} and the notions in \eqref{eq:auxi}. Moreover, let $\hat{\bm U}_k$ be as defined in the proof of \Cref{lem:sing matrix} for each $k \in [K]$. Now, we construct 
    \begin{align*}
        \hat{\bm Z} = \begin{bmatrix}
            \hat{\bm Z}_1 & \dots & \hat{\bm Z}_K
        \end{bmatrix},\ \text{where}\ \hat{\bm Z}_k = \sigma_k^* \hat{\bm U}_k \bm Q_k^{(1)^T},
    \end{align*}
    where $\bm Q_k^{(1)}$ is defined in \eqref{eq:auxi}. By \Cref{lem:sing matrix} and \Cref{fact:land}(ii), $\hat{\bm Z}\in\mathcal{Z}$. Moreover, for all $k \neq l \in [K]$, 
    \begin{align*}
        \hat{\bm Z}_k^T\hat{\bm Z}_l = \sigma_k^*\sigma_l^*\bm Q_k^{(1)^T} \hat{\bm U}_k^T\hat{\bm U}_l\bm Q_l^{(1)} = \bm 0, 
    \end{align*}
    where the last equality is due to \Cref{lem:sing matrix}. For each $k \in [K]$, we have 
    \begin{align*}
        \left\| \bm Z_k - \hat{\bm Z}_k \right\|_F & \overset{\eqref{eq:SVD Zk}}{=} \left\| \bm P_k^{(1)} \bm \Sigma_k^{(1)} \bm Q_k^{(1)^T} + \bm P_k^{(2)} \bm \Sigma_k^{(2)} \bm Q_k^{(2)^T} - \sigma_k^* \hat{\bm U}_k \bm Q_k^{(1)^T}\right\|_F \\
        & \le \left\| \bm P_k^{(1)} \bm \Sigma_k^{(1)} \bm Q_k^{(1)^T} - \sigma_k^* \hat{\bm U}_k \bm Q_k^{(1)^T}\right\|_F + \left\| \bm P_k^{(2)} \bm \Sigma_k^{(2)} \bm Q_k^{(2)^T} \right\|_F \\
        & = \left\| \bm P_k^{(1)} \bm \Sigma_k^{(1)}  - \sigma_k^* \hat{\bm U}_k \right\|_F + \left\| \bm \Sigma_k^{(2)} \right\|_F \\
        & \le  \left\| \bm P_k^{(1)} \left(\bm \Sigma_k^{(1)}  - \sigma_k^* \bm I\right) \right\|_F +  \sigma_k^* \left\| \bm P_k^{(1)} - \hat{\bm U}_k \right\|_F + \left\| \bm \Sigma_k^{(2)} \right\|_F \\
        & \le \left( \frac{(1+\kappa_1)(\alpha+\lambda)}{\kappa_2\kappa_3} \sqrt{r_k} + \frac{3^{k-2}\cdot 8K(1+\kappa_1)\sigma_k^*}{\kappa_2\alpha\sigma_{\min}^*} + \frac{2\sqrt{d}}{\lambda} \right) \|\nabla F(\bm Z)\|_F,
    \end{align*}
    where the last inequality uses \eqref{eq:eb sing 1} in \Cref{lem:bound sing}, \eqref{eq:bound sing 2} in \Cref{prop:bound sing}, and \eqref{eq:lem sing matrix} in \Cref{lem:sing matrix}. Therefore, we have
    \begin{align*}
        \mathrm{dist}(\bm Z, \mathcal{Z}) & \le \left\| \bm Z - \hat{\bm Z} \right\|_F = \sqrt{\sum_{k=1}^K \|\bm Z_k-\hat{\bm Z}_k\|_F^2}  \\
        & \le \left( \frac{(1+\kappa_1)(\alpha+\lambda)}{\kappa_2\kappa_3} \sum_{k=1}^K \sqrt{r_k} + \frac{3^{K-2}\cdot 8K^2(1+\kappa_1)\sigma_{\max}^*}{\kappa_2\alpha\sigma_{\min}^*} + \frac{2\sqrt{d}K}{\lambda} \right) \|\nabla F(\bm Z)\|_F. 
    \end{align*}
\end{proof}
% Furthermore, by \Cref{lem:nabla_lip}, both~\eqref{eq:dist 3} and~\eqref{eq:grad2} hold provided that
% \begin{align*}
%     \|\bm{Z} - \bm{Z}^*\| \le 
%     \min\left\{ 
%         \sqrt{\frac{\lambda}{2\alpha\,\alpha_{\min}}},\,
%         \frac{\sigma_{\min}^*}{2},\,
%         \frac{\sigma_{\min}^*}{8\sigma_{\max}^*},\,
%         \delta_1,\,
%         \kappa_4
%     \right\},
% \end{align*}
% where
% \begin{align*}
%     \kappa_4 
%     &:= \frac{1}{K\bigl(\lambda + 2\alpha + 9\alpha\alpha_{\max} (\sigma_{\max}^*)^2 \bigr)} \\
%     &\quad \cdot \min\left\{ 
%         \frac{\lambda\,\sigma_{\max}^*}{4\sqrt{dK}},\,
%         \frac{\lambda\,\sigma_{\min}^*}{16\sigma_{\max}^*\sqrt{dK}},\,
%         \frac{2\kappa_2\,\sigma_{\min}^*}{3^{K-2}\,4K(1+\kappa_1)}
%     \right\}.
% \end{align*}

%%%%%%%%%%%%%%%%%%%%%%%%%%%%%%%%%%%%%%%%%%%%%%%%%%%%%%%%%%%%%%%%%%%%%%%%%%%%%%%
%%%%%%%%%%%%%%%%%%%%%%%%%%%%%%%%%%%%%%%%%%%%%%%%%%%%%%%%%%%%%%%%%%%%%%%%%%%%%%%

% \section{Proof in \Cref{sec:main}}\label{app:main}

% \begin{proof}[Proof of \Cref{coro:rate}]
    
% \end{proof}

\section[Proofs in Section 2]{Proofs in \Cref{sec:main}}\label{app:main}

\subsection[Proof of Theorem 1]{Proof of \Cref{thm:EB}}\label{app:thm EB}

Armed with the above setup, we now proceed to prove \Cref{thm:EB} by applying the above argument to all points in the neighborhood $\cal Z$. 
\begin{proof}[Proof of \Cref{thm:EB}]
        According to \eqref{set:Z}, it is obvious that $\mathcal{Z}$ is a compact set. Let $\bm Z^* \in \cal Z$ be arbitrary. Using \Cref{prop:eb}, there exists a $\delta_{\bm Z^*} > 0$, defined in \eqref{eq:deltaZ}, such that for all $\|\bm Z - \bm Z^*\|_F \le \delta_{\bm Z^*}$, it holds that  
    \begin{align*}
        \mathrm{dist}(\bm Z, \mathcal{Z}) \le \kappa_{\bm Z^*} \|\nabla F(\bm Z)\|_F,
    \end{align*}
    where $\kappa_{\bm Z^*}$ is defined as in \eqref{eq:kappaZ}. Then, it always holds that 
    \begin{align*}
        \bigcup_{\bm Z^* \in \mathcal{Z}} \left\{ \bm Z: \|\bm Z - \bm Z^*\|_F \le \delta_{\bm Z^*} \right\}  \supseteq\mathcal{Z}.
    \end{align*}
    As $\cal Z$ is compact, using the Heine--Borel theorem, we can select finite elements $\bm Z^{(1)},\ldots,\bm Z^{(n)} \in \mathcal{Z}$ and the associated neighborhoods $\mathcal{U}_1,\ldots,\mathcal{U}_n$ such that $\bigcup_{i\in[n]}\mathcal{U}_{i}\supseteq \mathcal{Z}.$
    % By letting $\eta_1,\ldots,\eta_n$ be the corresponding coefficients,
    For each $i \in [n]$, it holds that  
    \begin{align*}
        \mathrm{dist}(\bm Z, \mathcal{Z})\le \kappa_{\bm Z^{(i)}} \|\nabla F(\bm Z)\|_F,\ \forall \bm Z \in \mathcal{U}_i.
    \end{align*}
    Therefore, we have
    \begin{align*}
        &\mathrm{dist}(\bm Z, \mathcal{Z})\le \kappa \|\nabla F(\bm Z)\|_F,\ \forall \bm Z \in \mathcal{U},
    \end{align*}
    where $\kappa := \max_{i\in[n]}\kappa_{\bm Z^{(i)}}$ and $\mathcal{U}:=\bigcup_{i\in[n]}\mathcal{U}_{i}$. 
\end{proof}

\begin{remark}
\label{remark:separated}
From \Cref{thm:EB}, we observe that if $\mathrm{dist}(\bm{Z}, \mathcal{Z}) \le \delta$ and $\nabla F(\bm{Z}) = \bm{0}$, then $\bm{Z} \in \mathcal{Z}$. 
This, together with \eqref{eq:EB}, implies that the set of local maximizers $\mathcal{Z}$ is well separated from strict saddle points, even though no explicit characterization of the latter is available.
\end{remark}

\subsection[Proof of Theorem 2]{Proof of \Cref{coro:rate}}\label{subsec:rate}

\begin{proof}[Proof of \Cref{coro:rate}]
First, we claim that Lyapunov stability holds, i.e., there exists a neighborhood $\mathcal{U}$ such that if 
$\bm{Z}^{t_0}\in \mathcal{U}$, then  $\bm{Z}^{t}\in \mathcal{U}$ for all $t \ge t_0$. We prove this claim by adapting \cite[Proposition~3.3]{boumal2023introduction} to the specific structure of our problem. By \Cref{fact:land}(ii), $\mathcal{Z}$ has finitely many connected components, determined by the choice of $\bm{r}$. 
This, together with \Cref{remark:separated}, implies that any connected component (say, $\mathfrak{X}$) is isolated from other critical points; that is, there exists a neighborhood $\mathcal{H}$ of $\mathfrak{X}$ such that  $\bm Z \in \mathcal{H}$ and $\nabla F(\bm Z) =\bm 0 $ imply $\bm Z \in \mathfrak{X}$ and $F(\cdot)$ is constant on $\mathfrak{X}$. We use $F(\bm Z^*)$ to denote the function value in this set for arbitrary $\bm Z^*\in \mathfrak{X}$.  Since $\mathcal{Z}$ has only finitely many connected components, we may, without loss of generality, restrict our attention to the  component $\mathfrak{X}$. Using \Cref{thm:EB} and \cite[Proposition 2.4]{boumal2023introduction}, there exists a neighborhood of $\mathfrak{X}$ such that
\begin{align}
\label{eq:plinequality}
 F(\bm Z^*)-F(\bm Z)\le \kappa \|\nabla F(\bm Z)\|_F^2.
\end{align}
Without loss of generality, we assume that this neighborhood is one on which the error bound inequality holds (see \Cref{thm:EB}), since we can always take the intersection of the two neighborhoods.
Because $\mathfrak{X}$ consists of local maximizers and is compact, there exists a compact neighborhood $\mathcal{V}$ of $\mathfrak{X}$ such that $F(\bm Z)< F(\bm Z^*)$ for each $\bm Z \in \mathcal{V}\setminus \mathfrak{X}$. By property~(i) of the generated sequence, we conclude $\lim_{t\to \infty} \|\bm Z^{t+1}-\bm Z^t\|_F=0$. 
Thus, there exists an open neighborhood $\mathcal{W}$ of $\mathfrak{X}$ such that $\mathcal{W} \subseteq \mathcal{V}\cap\mathcal{B}(\mathfrak{X}, \delta)$, where $\delta$ is as defined in \Cref{thm:EB}, and there exists $t_0$ such that $\bm{Z}^t \in \mathcal{W}$ implies $\bm{Z}^{t+1} \in \mathcal{V}$ for sufficiently large $t\ge t_0$. %Without loss of generality, we assume that $\mathcal{W} \subseteq \mathcal{B}(\mathfrak{X}, \delta)$, where $\delta$ is as defined in \Cref{thm:EB}; otherwise, we replace $\mathcal{W}$ by a smaller ball contained in it. 
Let $\alpha := \max_{\bm{Z} \in \mathcal{V} \setminus \mathcal{W}} F(\bm{Z})$, which is well defined because $\mathcal{V} \setminus \mathcal{W}$ is compact and $F$ is continuous. Define $\mathcal{U} := \{ \bm{Z} \in \mathcal{W} : F(\bm{Z}) > \alpha \}$. By construction, we have $\mathcal{U} \subseteq \mathcal{W}\subseteq \mathcal{B}(\mathfrak{X}, \delta)$. Note that $\bm{Z}^{t_0} \in \mathcal{U}$. Then we have $\bm{Z}^{t_0+1} \in \mathcal{V}$ by the above implication, and property~(i) gives $F(\bm{Z}^{t_0+1}) \ge F(\bm{Z}^{t_0}) > \alpha$.
This implies $\bm{Z}^{t_0+1} \in \mathcal{U}$.
% Hence $\bm{Z}^{t_0+1} \in \mathcal{U}$; otherwise, if $\bm{Z}^{t_0+1} \in \mathcal{V} \setminus \mathcal{U}$, we would have $F(\bm{Z}^{t_0+1}) \le \alpha$, contradicting $F(\bm{Z}^{t_0+1}) > \alpha$.
By induction, it follows that $\bm{Z}^t \in \mathcal{U}$ for all $t \ge t_0$. 
% Next, we show that there exists an index $t_0$ such that $\bm{Z}^{t_0} \in \mathcal{U}$.  
% By property~(i), we have
% \[
% \max_{\bm{Z} \in \mathbb{R}^{d \times N}} F(\bm{Z}) - F(\bm{Z}^1)
% \ge \sum_{t=1}^{\infty} \bigl( F(\bm{Z}^{t+1}) - F(\bm{Z}^t) \bigr)
% \ge \sum_{t=1}^{\infty} \mu_1 \|\bm{Z}^{t+1} - \bm{Z}^t\|_F^2.
% \]
% Hence, the series $\sum_{t=1}^\infty \|\bm{Z}^{t+1} - \bm{Z}^t\|_F^2$ converges, which implies 
% $\lim_{t \to \infty} \|\bm{Z}^{t+1} - \bm{Z}^t\|_F = 0$.  
% Combining this with property~(iii), we obtain $\lim_{t \to \infty} \|\nabla F(\bm{Z}^t)\|_F = 0$.
% From \Cref{eq:grad Fk}, we also know that $\|\nabla F(\bm{Z})\|_F \to \infty$ as $\|\bm{Z}\|_F \to \infty$ .  
% Therefore, the sequence $\{\bm{Z}^t\}$ must be bounded.  
% Since $\mathcal{Z}$ is compact and $\nabla F$ is continuous, any limit point of $\{\bm{Z}^t\}$ is a critical point, and because $\|\nabla F(\bm{Z}^t)\|_F \to 0$, every limit point lies in $\mathcal{Z}$.  
% Consequently, $\mathrm{dist}(\bm{Z}^t, \mathcal{Z}) \to 0$ as $t \to \infty$.  
% It follows that there exists some $t_0$ such that $\bm{Z}^{t_0} \in \mathcal{U}$.

According to \Cref{thm:EB} and $\bm{Z}^{t_0} \in \mathcal{U}$, we have for all $t \ge t_0$ that 
\begin{align}
\label{eq:eqeb}
\bm Z^{t}\in \mathcal{U}\quad\text{and}\quad\mathrm{dist}(\bm Z^{t}, \mathfrak{X}) \le \kappa \|\nabla F(\bm Z^{t})\|_F. 
\end{align}
Thus, we have
\begin{align*}
    F(\bm Z^*)-F(\bm Z^{t+1}) & \le \kappa\|\nabla F(\bm Z^{t+1})\|_F^2 \le \kappa L_{F}^2\mathrm{dist}^2(\bm Z^{t+1}, \mathfrak{X})\\
    & \le 2\kappa L_{F}^2\left(\mathrm{dist}^2(\bm Z^{t}, \mathfrak{X})+\|\bm Z^t -\bm Z^{t+1}\|_F^2\right)\\
    & \le 2\kappa L_{F}^2(\kappa^2\mu_2^2+1)\|\bm Z^t -\bm Z^{t+1}\|_F^2,
\end{align*}
where the first inequality follows from~\eqref{eq:plinequality}, the second from \Cref{lem:nabla_lip}, and the last from~\eqref{eq:eqeb}. Using (i), we have 
\[
\frac{F(\bm Z^*) - F(\bm Z^t)}{F(\bm Z^*) - F(\bm Z^{t+1})} \ge 1 + \frac{\mu_1 \|\bm Z^{t+1} - \bm Z^t\|_F^2}{F(\bm Z^*) - F(\bm Z^{t+1})}.
\]
% Next, combining the error bound \eqref{eq:EB}, (ii), and (iii), we obtain that
% \[
% F(\bm Z^*) - F(\bm Z^{t+1}) \le \mu_2 \left( \mathrm{dist}^2(\bm Z^t, \mathcal{Z}) + \|\bm Z^{t+1} - \bm Z^t\|_F^2 \right)
% \le \mu_2 (\kappa^2 \mu_3^2 + 1) \|\bm Z^{t+1} - \bm Z^t\|_F^2.
% \]
Substituting this into the inequality above, we have
\begin{align}
\label{eq:flinear}
\frac{F(\bm Z^*) - F(\bm Z^t)}{F(\bm Z^*) - F(\bm Z^{t+1})} \ge w^{-2},\quad \text{where}\ w :=\left( 1 + \frac{\mu_1}{2\kappa L_{F}^2(\kappa^2\mu_2^2+1)}\right)^{-1/2}.
\end{align}
% This implies for all $t\ge t_0+1$,
% \begin{align*}
%     F(\bm Z^*) - F(\bm Z^{t}) \le  w^{2(t-t_0)}\left( F(\bm Z^*) - F(\bm Z^{t_0}) \right).
% \end{align*}
Using (i) again, we have
\[
\mu_1 \|\bm Z^{t+1} - \bm Z^t\|_F^2 \le F(\bm Z^{t+1}) - F(\bm Z^t) \le F(\bm Z^{*}) - F(\bm Z^t) \le w^{2(t- t_0)} \left( F(\bm Z^{*}) - F(\bm Z^{t_0}) \right).
\]
Thus, we obtain that
\[
\|\bm Z^{t+1} - \bm Z^t\|_F\le \left( \frac{F(\bm Z^{*}) - F(\bm Z^{t_0})}{\mu_1} \right)^{1/2} w^{t - t_0}.
\]
For each \(m \ge t\), we have the following bound:
\[
\|\bm Z^t - \bm Z^m\|_F \le \left( \frac{F(\bm Z^{*}) - F(\bm Z^t)}{\mu_1} \right)^{1/2} \sum_{i=t}^{m-1} w^{i - t_0}.
\]
This shows that the sequence \(\{\bm Z^t\}\) is a Cauchy sequence. Thus, we conclude that $\{\bm{Z}^t\}$ converges to a local maximizer $\bar{\bm{Z}} \in \mathcal{Z}$. Note that $\bar{\bm{Z}}$ may differ from $\bm{Z}^*$, although $F(\bar{\bm{Z}}) = F(\bm{Z}^*)$.
Finally, letting \(m \to \infty\), we obtain that
\[
\|\bm Z^t - \bar{\bm Z}\|_F \le \left( \frac{F(\bm Z^{*}) - F(\bm Z^t)}{\mu_1} \right)^{1/2} \frac{w^{t - t_0}}{1 - w}.
\]
This, together with \eqref{eq:flinear}, implies the desired result. Since the finitely many connected components of $\mathcal{Z}$ are mutually separated, the conclusion extends from a neighborhood of a single component $\mathfrak{X}$ to a neighborhood of the entire set $\mathcal{Z}$. 
With a slight abuse of notation, we continue to denote this Lyapunov-stable neighborhood by $\mathcal{U}$. 
% This, together with 
% Thus, we conclude that the sequence \(\{F(\bm Z^t)\}_{t \ge t_0}\) converges \(Q\)-linearly to \(F(\bm Z^*)\) and \(\{\bm Z^t\}_{t \ge t_0}\) converges \(R\)-linearly to \(\mathcal{Z}\).
\end{proof}

\subsection[Proof of Corollary 1]{Proof of \Cref{coro:GD}}\label{subsec:coro GD}

\begin{proof}[Proof of \Cref{coro:GD}]
Combined with the fact that gradient ascent with random initialization almost surely escapes strict saddles~\cite{lee2019first} and the loss landscape of Problem~\eqref{eq:MCR} (see \Cref{fact:land}), the iterates almost always converge to a local maximizer. Moreover, since $\|\bm{Z}^{t+1} - \bm{Z}^t\|_F = \eta \|\nabla F(\bm{Z}^t)\|_F$ and the update norm vanishes asymptotically along the convergent trajectory, it follows from the argument in \cite[Proposition 3.3]{boumal2023introduction} that once $\bm{Z}^{t_0} \in \mathcal{U}$, which is defined in the proof of \Cref{coro:rate}, we have $\bm{Z}^t \in \mathcal{U}$ for all $t \ge t_0$. 

By \Cref{lem:nabla_lip}, 
\[
F(\bm{Z}^{t+1}) - F(\bm{Z}^t) \ge \langle \nabla F(\bm{Z}^t), \bm{Z}^{t+1} - \bm{Z}^t \rangle - \frac{L_F}{2} \|\bm{Z}^{t+1} - \bm{Z}^t\|_F^2 = \left( \frac{1}{\eta} - \frac{L_F}{2} \right) \|\bm{Z}^{t+1} - \bm{Z}^t\|_F^2.
\]
Assuming $\eta \le 1/L_F$, this implies $$F(\bm{Z}^{t+1}) - F(\bm{Z}^t) \ge \frac{1}{2\eta} \|\bm{Z}^{t+1} - \bm{Z}^t\|_F^2 \ge \frac{L_F}{2} \|\bm{Z}^{t+1} - \bm{Z}^t\|_F^2.$$ Combining this with \Cref{coro:rate} yields the desired result with
\[
w = \left( 1 + \frac{1}{4\kappa L_{F} (\kappa^2/\eta^2 + 1)} \right)^{-1/2}.
\]
Here we use the explicit choices $\mu_1=L_F/2$ and $\mu_2=1/\eta$ from the gradient-ascent update.
\end{proof}

% % BEGIN LOCAL OPENREVIEW REMARKS 2026-10-02
% {\color{blue}
% \begin{remark}\label{remark:parameter-transfer}
% The results concern direct optimization over the representation matrix. For a continuously differentiable feature map, let $\Phi(\theta)=[\varphi_\theta(\bm x_1),\ldots,\varphi_\theta(\bm x_N)]$, $G=F\circ\Phi$, $J_\theta=D\mathrm{vec}(\Phi)(\theta)$, and $g(\theta)=\mathrm{vec}(\nabla F(\Phi(\theta)))$. The chain rule gives $\nabla G(\theta)=J_\theta^Tg(\theta)$, so stationarity and error bounds need not transfer automatically to parameter space. Under \Cref{thm:EB}, fix a connected component $\mathfrak X$ of $\mathcal Z$ and $\bar\theta\in\Theta_{\mathfrak X}:=\Phi^{-1}(\mathfrak X)$. Suppose that, locally near $\bar\theta$, there exist $c_\Phi,\gamma>0$ such that
% \[
% \mathrm{dist}(\theta,\Theta_{\mathfrak X})\le c_\Phi\mathrm{dist}(\Phi(\theta),\mathfrak X),
% \qquad \|J_\theta^Tg(\theta)\|\ge\gamma\|g(\theta)\|.
% \]
% Component separation and \Cref{thm:EB} then yield the conditional parameter-space error bound
% \[
% \mathrm{dist}(\theta,\Theta_{\mathfrak X})
% \le c_\Phi\kappa\|g(\theta)\|
% \le \frac{c_\Phi\kappa}{\gamma}\|\nabla G(\theta)\|.
% \]
% The first condition controls the distance to parameters realizing $\mathfrak X$; the second prevents the Jacobian from suppressing the feature gradient. Thus, parameter-space guarantees require additional regularity of the feature map. The network experiments in \Cref{subsec:exp2,app:extension-experiments} illustrate the observed behavior without establishing these conditions.
% \end{remark}
% }

\begin{remark}\label{remark:local-scope}
\Cref{coro:rate} gives a local rate conditional on sufficiently late entry into its neighborhood. For gradient ascent under \Cref{coro:GD}, almost-sure convergence implies eventual entry; no uniform bound on the entry time or global linear rate is asserted. The local constants are characterized in \Cref{prop:eb}, specifically \eqref{eq:deltaZ}--\eqref{eq:kappaZ}. To clarify why the trajectory stays nearby, let $\mathfrak X$ be one connected component and $F^*$ its objective value. Sufficient ascent and boundedness of $F$ above imply $\Delta_t:=\bm Z^{t+1}-\bm Z^t\to\bm 0$. As in the proof, choose a compact neighborhood $\mathcal V$ within the error-bound region with $F<F^*$ on $\mathcal V\setminus\mathfrak X$, and an open neighborhood $\mathcal W$ of $\mathfrak X$ with $\overline{\mathcal W}\subset\mathrm{int}(\mathcal V)$. Set
\[
\rho:=\mathrm{dist}(\overline{\mathcal W},\mathcal V^c)>0,\qquad
a:=\max_{\mathcal V\setminus\mathcal W}F<F^*,\qquad
\mathcal U:=\{\bm Z\in\mathcal W:F(\bm Z)>a\}.
\]
For sufficiently large $t$, $\|\Delta_t\|_F<\rho$. If $\bm Z^t\in\mathcal U$, the next iterate lies in $\mathcal V$ and sufficient ascent gives $F(\bm Z^{t+1})>a$, excluding $\mathcal V\setminus\mathcal W$. Induction therefore keeps every subsequent iterate in $\mathcal U$. Entry into $\mathcal U$ remains a separate assumption for the general algorithmic class in \Cref{coro:rate}.
\end{remark}

\section{Auxiliary Results}

\begin{lemma}\label{lem:inverse_identity}
For a matrix $\bm A \in \mathbb{R}^{m \times n}$, it holds that 
\begin{align}\label{eq:inv}
    (\bm I_m + \bm{A}\bm{A}^T)^{-1}\bm{A} = \bm{A}(\bm I_n + \bm{A}^T \bm{A})^{-1}.
\end{align}
\end{lemma}
\begin{proof}
Let $r=\mathrm{rank}(\bm A)$ and $\bm A = \bm U\bm \Sigma\bm V^T$ be an SVD of $\bm A$, where $\bm U \in \mathcal{O}^m$, $\bm V \in \mathcal{O}^n$, and $\bm \Sigma = \mathrm{BlkD}(\tilde{\bm \Sigma},\bm 0) \in \mathbb{R}^{m \times n}$ with $\tilde{\bm \Sigma} \in \R^{r\times r}$ being a diagonal matrix. Then, we compute 
\begin{align}
    & \left(\bm I_m + \bm{A}\bm{A}^T\right)^{-1}\bm{A} = \bm U\left( \bm I_m + \bm \Sigma\bm \Sigma^T\right)^{-1} \bm \Sigma\bm V^T, \label{eq1}\\
    & \bm{A}(\bm I_n + \bm{A}^T \bm{A})^{-1} = \bm U \bm \Sigma \left( \bm I_n + \bm \Sigma^T\bm \Sigma\right)^{-1}\bm V^T. \label{eq2}
\end{align}
Note that 
\begin{align*}
     \left( \bm I_m + \bm \Sigma\bm \Sigma^T\right)^{-1} \bm \Sigma =  \begin{bmatrix}
         (\bm I_r + \tilde{\bm \Sigma}^2)^{-1} \tilde{\bm \Sigma} & \bm 0 \\
         \bm 0 & \bm 0
     \end{bmatrix},\quad  \bm \Sigma \left( \bm I_n + \bm \Sigma\bm \Sigma^T\right)^{-1} =  \begin{bmatrix}
         \tilde{\bm \Sigma}(\bm I_r + \tilde{\bm \Sigma}^2)^{-1}  & \bm 0 \\
         \bm 0 & \bm 0
     \end{bmatrix}. 
\end{align*}
This, together with $(\bm I_r + \tilde{\bm \Sigma}^2)^{-1} \tilde{\bm \Sigma} =\tilde{\bm \Sigma}(\bm I_r + \tilde{\bm \Sigma}^2)^{-1}$, \eqref{eq1}, and \eqref{eq2}, yields \eqref{eq:inv}.
\end{proof}

\begin{lemma}\label{lem:g}
    Given the coding precision $\epsilon > 0$, suppose that $\lambda$ satisfies
\begin{align}\label{eq:lambda 1}
\lambda \in \left(0, \frac{d(\sqrt{N/N_{\max}}-1)}{N ( \sqrt{N/N_{\max}} + 1)\epsilon^2} \right]. 
\end{align}
For each $k \in [K]$, define
\begin{align*}
    \varphi_k(x) := \frac{\alpha}{1 + \alpha x^2} - \frac{ \alpha }{1+\alpha_kx^2} - \lambda,
\end{align*}
where $\alpha$ and $\alpha_k$ are defined in \eqref{eq:alp}. Then, the following statements hold:\\
(i) For each $k \in [K]$, the equation $\varphi_k(x) = 0$ has two distinct positive roots
\begin{align}\label{eq:roots}
    & \sigma_k^* = \left(\frac{ \eta_k + \sqrt{\eta_k^2 - 4\lambda^2 N/N_k}}{2\lambda \alpha_k}\right)^{1/2},\ \sigma_{k}^\dagger =  \left(\frac{ \eta_k - \sqrt{\eta_k^2 - 4\lambda^2 N/N_k}}{2\lambda \alpha_k}\right)^{1/2}
\end{align}
where $\eta_k$ is defined in \eqref{eq:sigma eta} for each $k \in [K]$. 

(ii) For all $x$ satisfying $\left|x - \sigma_k^*\right| \le \delta_1$, it holds that 
\begin{align}\label{eq:gk}
    \left| \varphi_k(x)  \right| \ge \kappa_3\left| x - \sigma_k^* \right|, 
\end{align}
where 
\begin{align}\label{eq:kappa 3}
    \delta_1 := \frac{1}{2}\min\left\{\sigma_k^* - \sigma_{k}^\dagger: k \in [K]\right\},\ \kappa_3 := \frac{2\alpha^2\alpha_k(\sigma_k^* - \delta_1)(\alpha_k-\alpha)\sigma_k^* \sigma_{k}^\dagger\delta_1^2}{\left(1+\alpha(\sigma_k^* + \delta_1)^2\right)^2\left(1+\alpha_k(\sigma_k^* + \delta_1)^2\right)^2}.
\end{align}
\end{lemma}
\begin{proof}
    (i) Computing the roots of $\varphi_k(x) = 0$ is equivalent to solving 
    \begin{align*}
        \lambda\alpha_k x^4 - \eta_k x^2 + \frac{\lambda}{\alpha} = 0, 
    \end{align*}
    where $\eta_k$ is defined in \eqref{eq:sigma eta}. Using \eqref{eq:lambda 1}, one can verify that for each $k \in [K]$, 
    \begin{align*}
        \eta_k > 0,\ \eta_k^2 - \frac{4\alpha_k}{\alpha}\lambda^2 \ge 0. 
    \end{align*}
    This, together with \eqref{eq:alp}, yields that the above quadratic equation has positive roots as follows: 
    \begin{align*}
        x^2 = \frac{\eta_k \pm \sqrt{\eta_k^2 - 4\lambda^2N/N_k}}{2\lambda\alpha_k}, 
    \end{align*}
    which directly implies \eqref{eq:roots}. According to Vieta's formula, we have 
    \begin{align}\label{eq1:lem g}
        \left(\sigma_k^* \sigma_{k}^\dagger\right)^2 = \frac{1}{\alpha\alpha_k}.
    \end{align}

(ii) Let $x$ be arbitrary such that $\left|x - \sigma_k^*\right| \le \delta_1$. This, together with $\delta_1 = \frac{1}{2}\min\left\{\sigma_k^* - \sigma_{k}^\dagger: k \in [K]\right\}$, implies 
\begin{align}\label{eq2:lem g}
\frac{\sigma_k^* + \sigma_{k}^\dagger}{2} \le \sigma_k^* - \delta_1 \le x \le \sigma_k^* + \delta_1. 
\end{align}
We compute
\begin{align*}
    \left|\varphi_k^\prime(x)\right| & = 2\alpha \left| \frac{\alpha x}{\left(1+\alpha x^2\right)^2} - \frac{\alpha_k x}{(1+\alpha_k x^2)^2} \right| \\
    & = \frac{2\alpha x}{\left(1+\alpha x^2\right)^2(1+\alpha_k x^2)^2} \left|(\alpha - \alpha_k)(1-\alpha\alpha_kx^4)\right| \\
    & = \frac{2\alpha^2\alpha_k x}{\left(1+\alpha x^2\right)^2(1+\alpha_k x^2)^2} \left|(\alpha - \alpha_k)\left(\left(\sigma_k^* \sigma_{k}^\dagger\right)^2 - x^4\right)\right| \\
    & \ge \frac{2\alpha^2\alpha_k(\sigma_k^* - \delta_1)(\alpha_k-\alpha)\sigma_k^* \sigma_{k}^\dagger}{\left(1+\alpha(\sigma_k^* + \delta_1)^2\right)^2\left(1+\alpha_k(\sigma_k^* + \delta_1)^2\right)^2}\left(x^2 - \sigma_k^* \sigma_{k}^\dagger \right) \\
    & \ge \frac{2\alpha^2\alpha_k(\sigma_k^* - \delta_1)(\alpha_k-\alpha)\sigma_k^* \sigma_{k}^\dagger}{\left(1+\alpha(\sigma_k^* + \delta_1)^2\right)^2\left(1+\alpha_k(\sigma_k^* + \delta_1)^2\right)^2} \left( \left( \frac{\sigma_k^* + \sigma_{k}^\dagger}{2} \right)^2 - \sigma_k^* \sigma_{k}^\dagger  \right) \\
    & \ge \frac{2\alpha^2\alpha_k(\sigma_k^* - \delta_1)(\alpha_k-\alpha)\sigma_k^* \sigma_{k}^\dagger\delta_1^2}{\left(1+\alpha(\sigma_k^* + \delta_1)^2\right)^2\left(1+\alpha_k(\sigma_k^* + \delta_1)^2\right)^2},
\end{align*}
where the third equality uses \eqref{eq1:lem g}, the first and second inequalities follow from \eqref{eq2:lem g}, and the last inequality uses the definition of $\delta_1$. This, together with the mean value theorem and $ \varphi_k(\sigma_k^*) = 0$, yields \eqref{eq:gk}. 
\end{proof} 
\section{Extension Experiments}\label{app:extension-experiments}

\Cref{subsec:exp2} studies the qualitative two-stage convergence behavior in a deterministic full-batch setting. Here we use a more practical minibatch training setup to examine whether progress on the MCR$^2$ objective remains aligned with representation quality. This experiment is complementary to the convergence plots in the main text: it tracks the objective together with linear-probe accuracy and class-mean geometry along training.

\paragraph{Setup.}
We conduct supervised MCR$^2$ training on MNIST using a CNN encoder with a $64$-dimensional normalized feature output. The network is trained for $10$ epochs with batch size $512$, and we compare three optimizers: {\em AdaGrad}, {\em RMSProp}, and {\em AdamW}. The learning rates are set to $0.01$ for AdaGrad and $0.001$ for both RMSProp and AdamW, with weight decay $10^{-4}$; a cosine annealing scheduler is applied throughout training. To ensure a fair comparison, all optimizers use the same model initialization and the same minibatch order.

\paragraph{Objective value and linear-probe accuracy.}
During training, we record the MCR$^2$ objective value and linear-probe test accuracy. The trajectory in \Cref{fig:comparison} shows that optimizers that improve the MCR$^2$ objective more quickly also tend to obtain stronger downstream linear-probe accuracy. \Cref{tab:mnist_optimizer_comparison} reports the same trend at selected iterations. All methods start from the same initialization with identical accuracy, but AdamW consistently attains both a larger objective value and higher linear-probe accuracy during training.
\begin{figure}[th]
    \centering
    \begin{subfigure}[t]{0.48\textwidth}
        \centering
        \includegraphics[width=\textwidth]{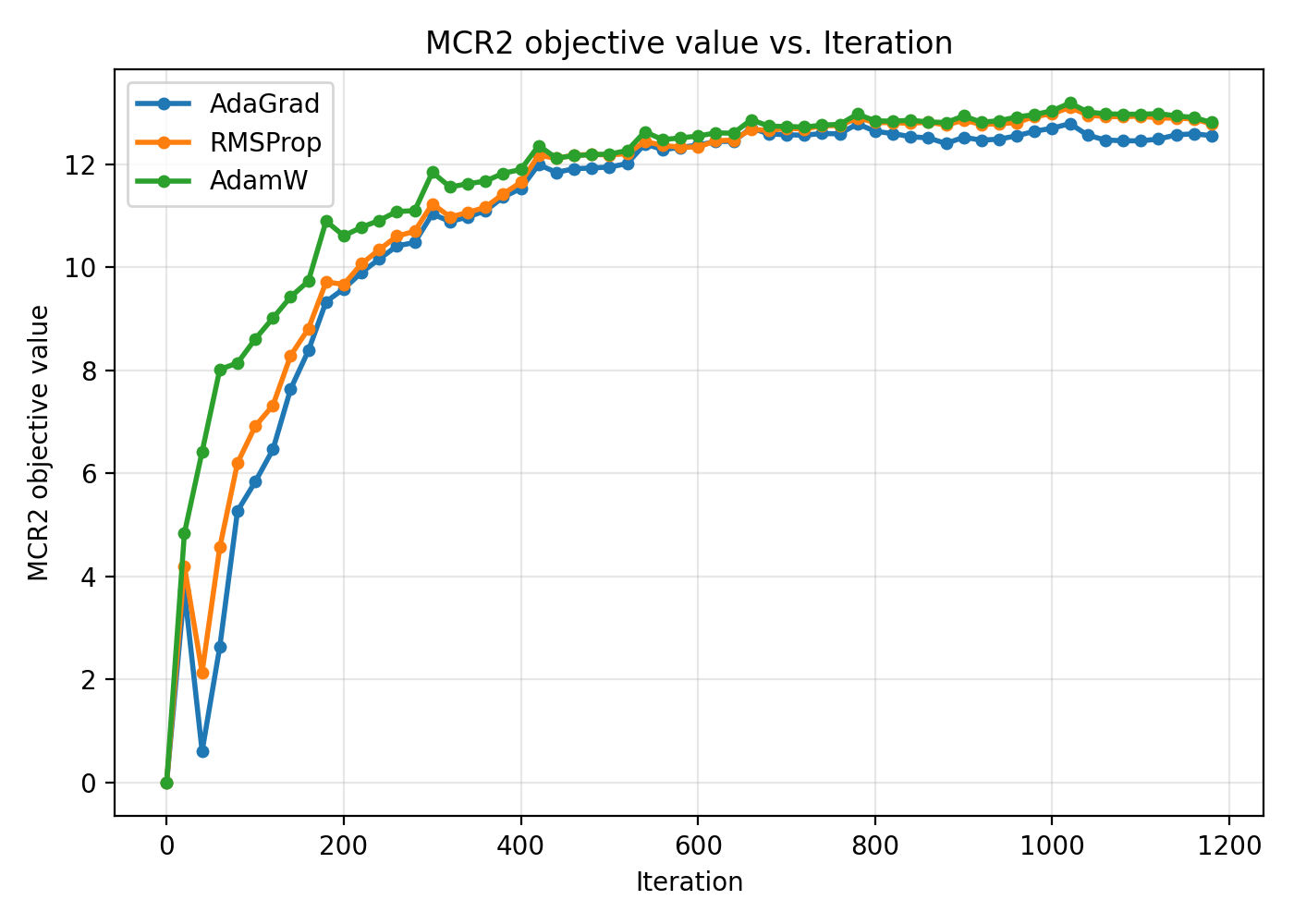}
        \caption{Objective value vs. iteration.}
        \label{fig:obj}
    \end{subfigure}
    \hfill
    \begin{subfigure}[t]{0.48\textwidth}
        \centering
        \includegraphics[width=\textwidth]{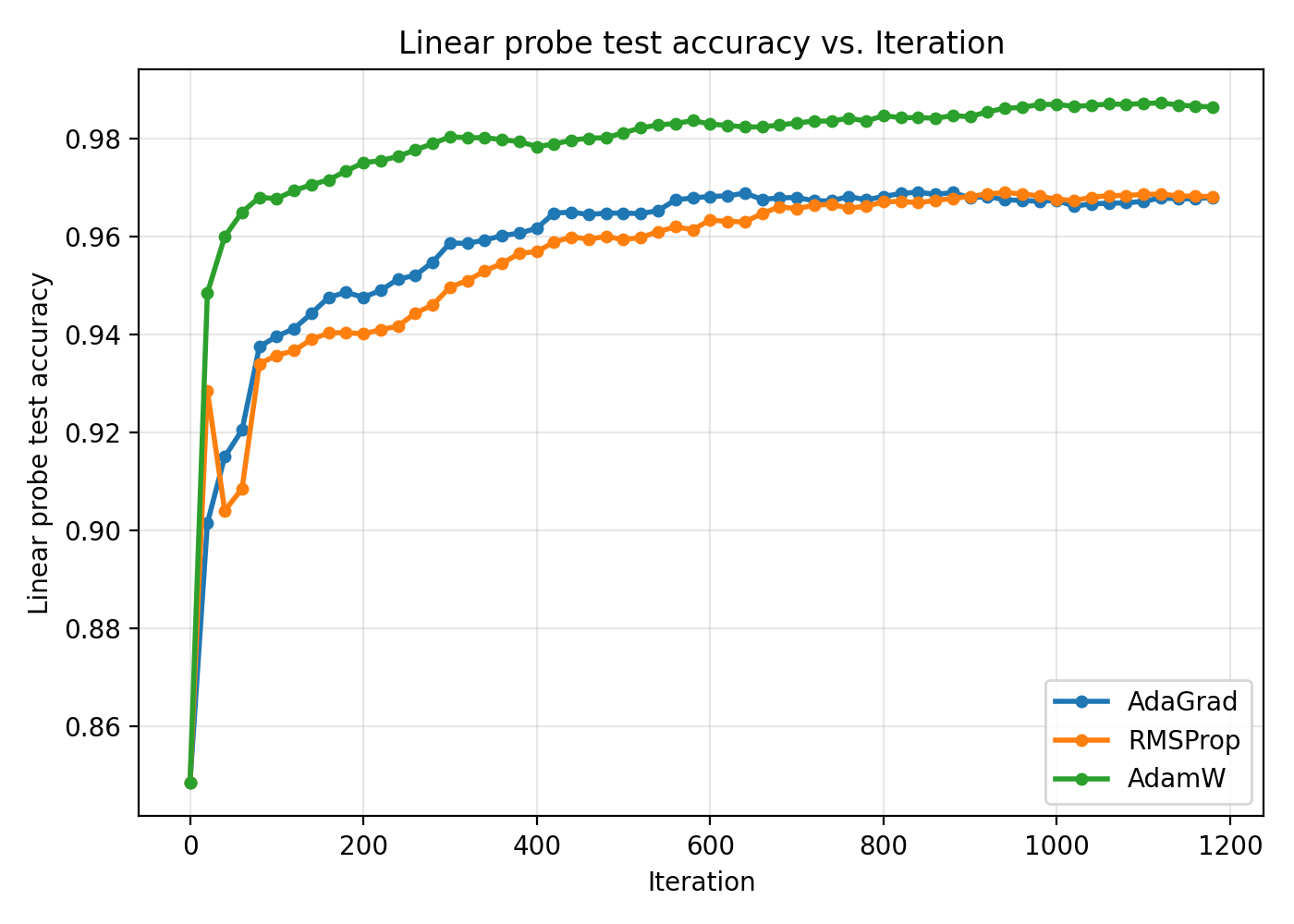}
        \caption{Linear-probe accuracy vs. iteration.}
        \label{fig:acc}
    \end{subfigure}
    \caption{Comparison of optimization behavior and representation quality across different optimizers.}
    \label{fig:comparison}
\end{figure}

\begin{table}[th]
\centering
\caption{Objective value and linear-probe test accuracy at selected iterations on MNIST.}
\label{tab:mnist_optimizer_comparison}
\begin{tabular}{c l c c}
\toprule
Iteration & Optimizer & Objective & Test Acc. \\
\midrule
0    & AdaGrad & 0.0009  & 84.85\% \\
0    & AdamW   & 0.0009  & 84.85\% \\
0    & RMSProp & 0.0009  & 84.85\% \\
\midrule
200  & AdaGrad & 9.8483  & 95.50\% \\
200  & AdamW   & 11.0919 & 97.50\% \\
200  & RMSProp & 10.0209 & 93.70\% \\
\midrule
400  & AdaGrad & 11.7208 & 96.35\% \\
400  & AdamW   & 12.2218 & 97.95\% \\
400  & RMSProp & 11.8589 & 95.80\% \\
\midrule
600  & AdaGrad & 11.5014 & 97.20\% \\
600  & AdamW   & 11.6748 & 98.30\% \\
600  & RMSProp & 11.8196 & 96.10\% \\
\midrule
800  & AdaGrad & 12.7542 & 97.00\% \\
800  & AdamW   & 13.0624 & 98.60\% \\
800  & RMSProp & 12.8779 & 96.55\% \\
\midrule
1000 & AdaGrad & 12.9536 & 96.70\% \\
1000 & AdamW   & 13.2556 & 98.65\% \\
1000 & RMSProp & 13.1013 & 96.65\% \\
\bottomrule
\end{tabular}
\end{table}

\paragraph{Feature separability.}
To examine feature separability, we compute cosine-similarity heatmaps of class-mean feature vectors using up to $5000$ randomly selected training samples, where each class mean is obtained by averaging normalized features within that class. We plot these heatmaps at selected iterations to visualize how the class structure evolves during optimization in \Cref{fig:similarity}. As optimization proceeds, the off-diagonal entries become smaller and the heatmap becomes closer to diagonal, indicating that class-mean features become more separated. This is consistent with the local maximizer structure in \Cref{fact:land}(ii), where different class subspaces are orthogonal.
\begin{figure}[th]
    \centering
    \begin{subfigure}[t]{0.32\textwidth}
        \centering
        \includegraphics[width=\textwidth]{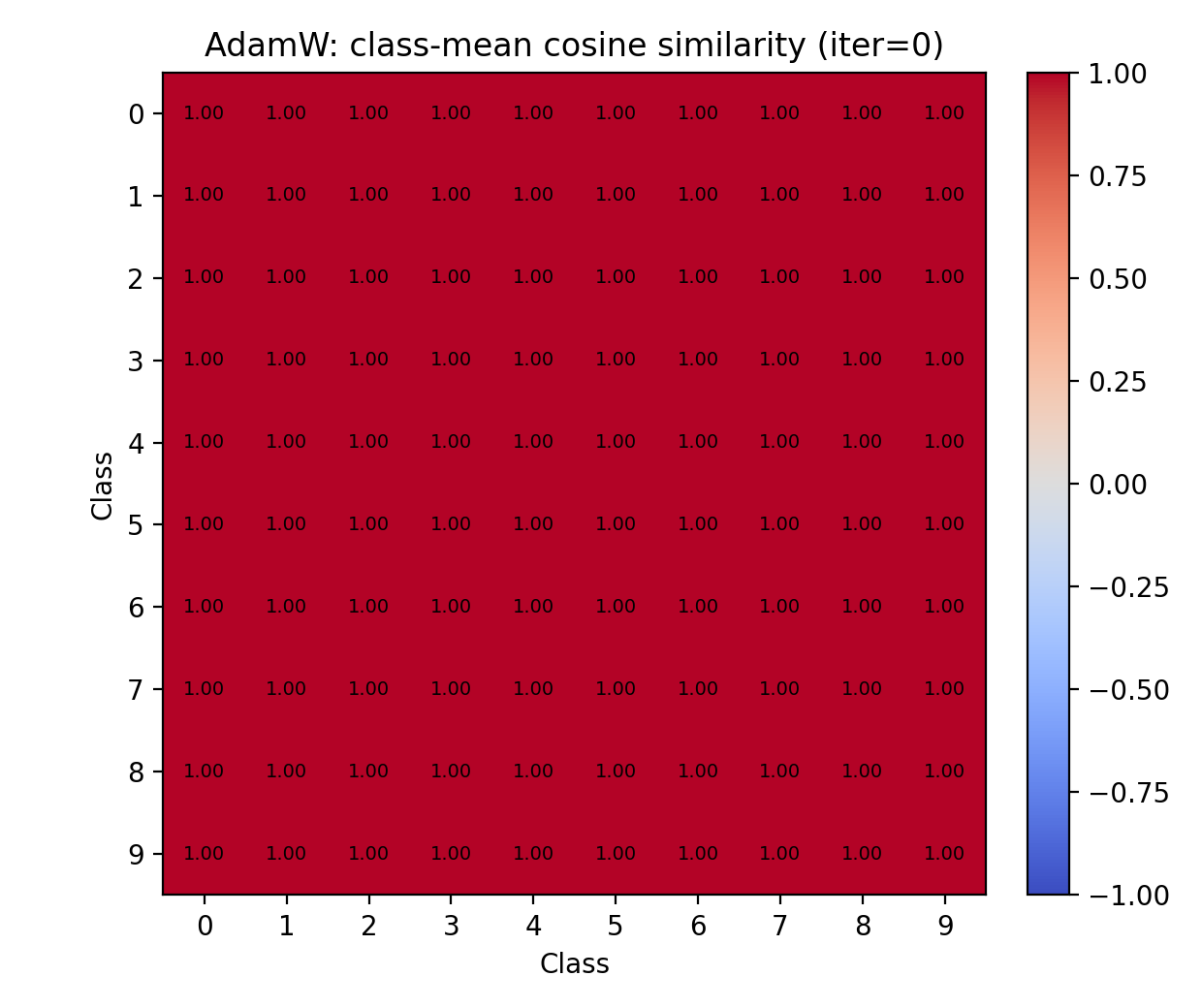}
        % \caption{Iteration 0}
        % \label{fig:obj}
    \end{subfigure}
    \hfill
    \begin{subfigure}[t]{0.32\textwidth}
        \centering
        \includegraphics[width=\textwidth]{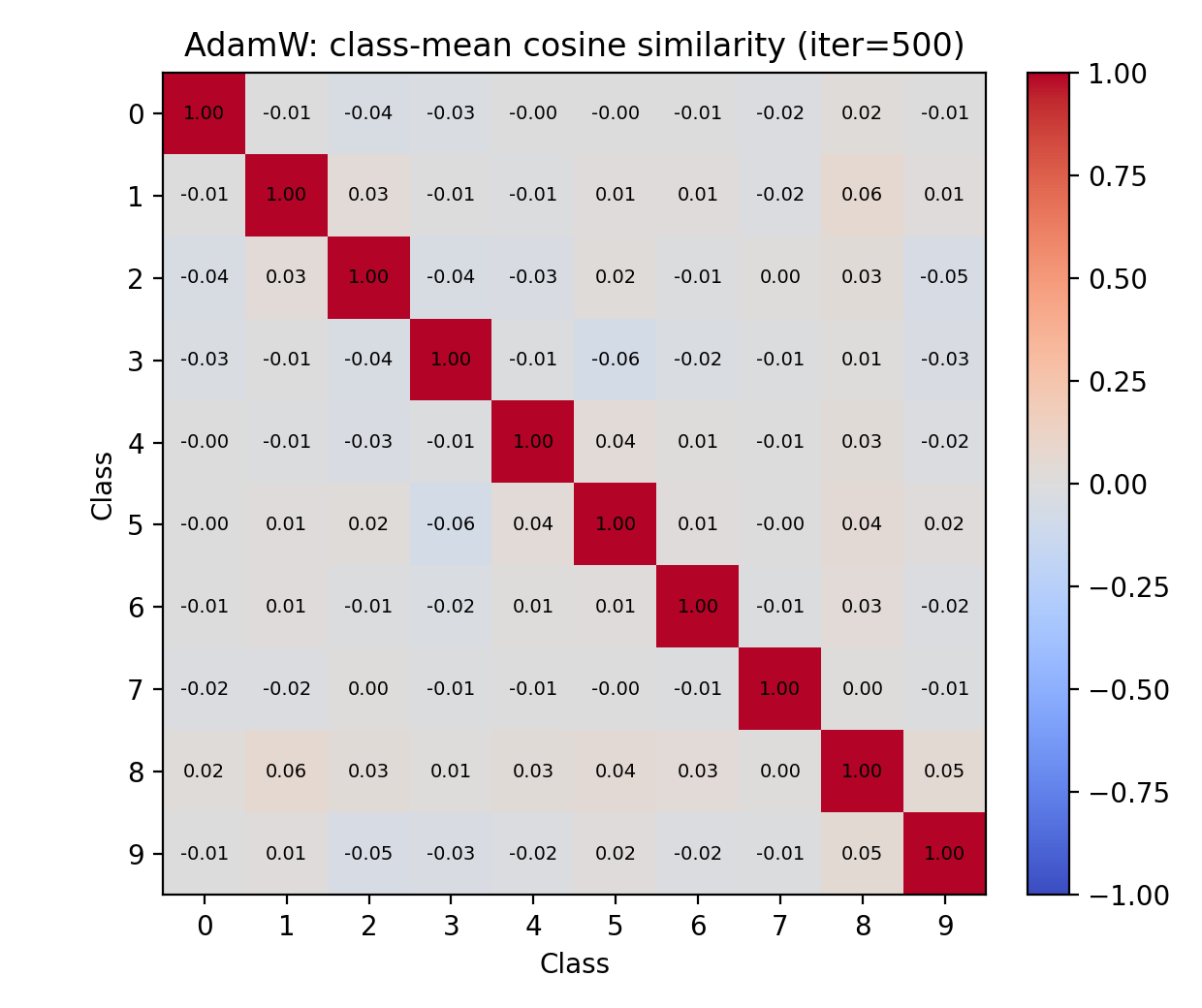}
        % \caption{Linear-probe accuracy vs. iteration.}
        % \label{fig:acc}
    \end{subfigure}
        \begin{subfigure}[t]{0.32\textwidth}
        \centering
        \includegraphics[width=\textwidth]{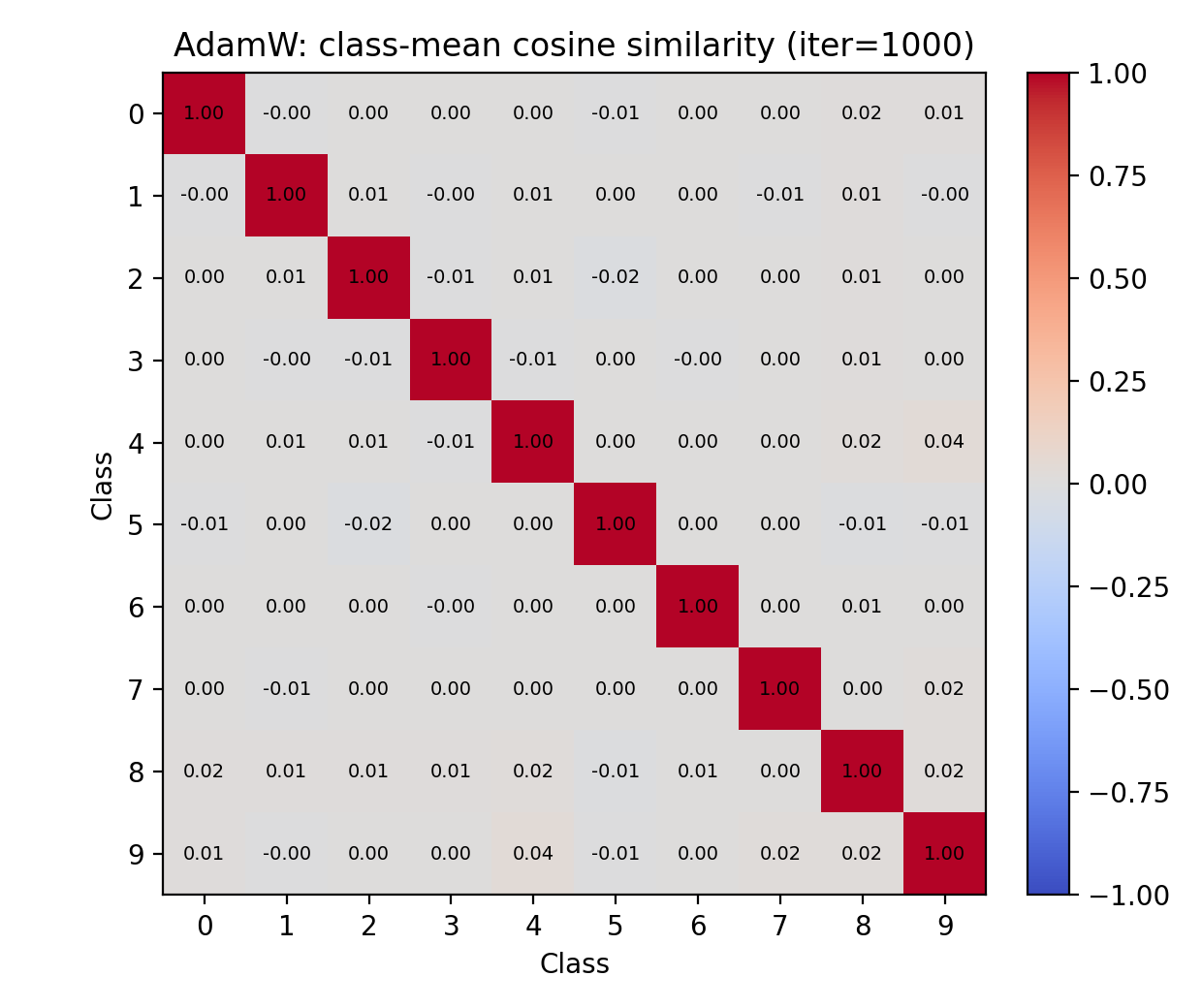}
        % \caption{Linear-probe accuracy vs. iteration.}
        % \label{fig:acc}
    \end{subfigure}
    \caption{Evolution of the class-mean cosine similarity matrix on MNIST for AdamW.}
    \label{fig:similarity}
\end{figure}

Overall, these extension experiments show a limited optimization--geometry alignment: in this controlled minibatch setting, faster progress on the MCR$^2$ objective is accompanied by improved linear-probe accuracy and clearer class-mean separation. The theorem itself concerns the MCR$^2$ objective, and these empirical results should be viewed as a representation-facing diagnostic rather than a broad benchmark comparison.

\paragraph{Reproducibility, variability, and compute.}
All experiments use standard public datasets. MNIST and CIFAR-10 are accessed through their canonical public releases, and the synthetic experiment generates Gaussian random initializations with the dimensions and hyperparameters specified in \Cref{subsec:exp1}. For the convergence plots, the main reported quantities are deterministic optimization trajectories under the stated initialization, objective, step size or learning-rate schedule, and stopping horizon. The deep-network convergence experiments in \Cref{subsec:exp2} are run in a deterministic full-batch setting to isolate the local linear regime. The optimizer comparison controls stochastic variability by using the same model initialization and the same minibatch order for all optimizers, so the reported trajectories are paired.

The reported experiments are lightweight relative to large-scale deep-learning benchmarks. The synthetic and full-batch convergence experiments can be reproduced on a single workstation, and the MNIST minibatch extensions can be run on a single commodity GPU without distributed training. An anonymized supplementary code archive accompanying the submission contains scripts for generating the synthetic data, running the MCR$^2$ optimization experiments, training the MNIST extension models, and reproducing the reported figures and tables.

\end{document}